\documentclass[a4paper,11pt]{article}
\usepackage{amsmath,amsthm,amssymb,enumitem}
\usepackage[mathscr]{eucal}
\usepackage[bookmarks=false]{hyperref}

\newlist{enumlist}{enumerate}{1}
\setlist[enumlist]{labelindent=0cm,label=\arabic*.,labelwidth=2.5ex,labelsep=0.5ex,leftmargin=3ex,align=left,topsep=0.5ex,itemsep=1ex,parsep=1ex}

\newlist{itemlist}{itemize}{1}
\setlist[itemlist]{labelindent=0cm,label=$\bullet$,labelwidth=2.5ex,labelsep=0.5ex,leftmargin=3ex,align=left,topsep=0.5ex,itemsep=1ex,parsep=1ex}

\numberwithin{equation}{section}

{\theoremstyle{definition}\newtheorem{definition}{Definition}[section]

\newtheorem{remark}[definition]{Remark}
}

\newtheorem{proposition}[definition]{Proposition}
\newtheorem{lemma}[definition]{Lemma}
\newtheorem{theorem}[definition]{Theorem}
\newtheorem{corollary}[definition]{Corollary}
\newtheorem{letterthm}{Theorem}

\newcommand{\bim}[3]{\mathord{\raisebox{-0.4ex}[0ex][0ex]{\scriptsize $#1$}{#2}\hspace{-0.25ex}\raisebox{-0.4ex}[0ex][0ex]{\scriptsize $#3$}}}

\newcommand{\coms}[4]{\begin{array}{ccc}
#1 & \subset & #2 \\ \cup & & \cup \\ #3 & \subset & #4
\end{array}}

\newcommand{\C}{\mathbb{C}}
\newcommand{\cC}{\mathcal{C}}
\newcommand{\eps}{\varepsilon}
\newcommand{\al}{\alpha}
\newcommand{\be}{\beta}
\newcommand{\albar}{\overline{\alpha}}
\newcommand{\End}{\operatorname{End}}
\newcommand{\Irr}{\operatorname{Irr}}
\newcommand{\Mor}{\operatorname{Mor}}
\newcommand{\Rep}{\operatorname{Rep}}
\newcommand{\ot}{\otimes}
\newcommand{\recht}{\rightarrow}
\newcommand{\cb}{_\text{\rm cb}}
\newcommand{\Z}{\mathbb{Z}}
\newcommand{\vphi}{\varphi}
\newcommand{\boxt}{\boxtimes}
\newcommand{\op}{^\text{\rm op}}
\newcommand{\bG}{\mathbb{G}}
\newcommand{\cO}{\mathcal{O}}
\newcommand{\id}{\mathord{\text{\rm id}}}
\newcommand{\om}{\omega}
\newcommand{\SU}{\operatorname{SU}}
\newcommand{\bebar}{\overline{\beta}}
\newcommand{\pibar}{\overline{\pi}}
\newcommand{\N}{\mathbb{N}}
\newcommand{\cL}{\mathcal{L}}
\newcommand{\ovt}{\mathbin{\overline{\otimes}}}
\newcommand{\mult}{\operatorname{mult}}
\newcommand{\Tr}{\operatorname{Tr}}
\newcommand{\real}{\operatorname{Re}}
\newcommand{\etabar}{\overline{\eta}}
\newcommand{\Om}{\Omega}
\newcommand{\bGhat}{\widehat{\mathbb{G}}}
\newcommand{\cD}{\mathcal{D}}
\newcommand{\si}{\sigma}
\newcommand{\R}{\mathbb{R}}
\newcommand{\PSU}{\operatorname{PSU}}
\newcommand{\counit}{\epsilon}
\newcommand{\Ups}{\Upsilon}
\newcommand{\F}{\mathbb{F}}
\newcommand{\cH}{\mathcal{H}}
\newcommand{\otalg}{\otimes_{\text{\rm alg}}}
\newcommand{\mubar}{\overline{\mu}}
\newcommand{\mupbar}{\overline{\mu'}}
\newcommand{\bbar}{\overline{b}}
\newcommand{\muppbar}{\overline{\mu''}}
\newcommand{\cZ}{\mathcal{Z}}

\newcommand{\cG}{\mathcal{G}}
\newcommand{\SE}{\text{\rm SE}}
\newcommand{\cK}{\mathcal{K}}
\newcommand{\cGtil}{\widetilde{\mathcal{G}}}
\newcommand{\cJ}{\mathcal{J}}
\newcommand{\Omtil}{\widetilde{\Omega}}
\newcommand{\counittil}{\widetilde{\counit}}
\newcommand{\cF}{\mathcal{F}}
\newcommand{\T}{\mathbb{T}}
\newcommand{\actson}{\curvearrowright}
\newcommand{\pT}{{\fontshape{n}\selectfont (T)}}
\newcommand{\bGh}{\widehat{\bG}}

\begin{document}

\begin{center}
{\boldmath\LARGE\bf Representation theory for subfactors,\vspace{0.5ex}\\ $\lambda$-lattices and C$^*$-tensor categories}

\bigskip

{\sc by Sorin Popa\footnote{Mathematics Department, UCLA, CA 90095-1555 (United States), popa@math.ucla.edu\\
Supported in part by NSF Grant DMS-1400208} and Stefaan Vaes\footnote{KU~Leuven, Department of Mathematics, Leuven (Belgium), stefaan.vaes@wis.kuleuven.be \\
    Supported by ERC Consolidator Grant 614195 from the European Research Council under the European Union's Seventh Framework Programme.}}

\bigskip

{\it Dedicated to Vaughan Jones}

%\bigskip
%
%{\it Communications in Mathematical Physics,} to appear.
\end{center}

\begin{abstract}\noindent
We develop a representation theory for $\lambda$-lattices, arising as standard invariants of subfactors, and for rigid C$^*$-tensor categories, including a definition of their universal C$^*$-algebra. We use this to give a systematic account of approximation and rigidity properties for subfactors and tensor categories, like (weak) amenability, the Haagerup property and property~(T). We determine all unitary representations of the Temperley-Lieb-Jones $\lambda$-lattices and prove that they have the Haagerup property and the complete metric approximation property. We also present the first subfactors with property~(T) standard invariant and that are not constructed from property~(T) groups.
\end{abstract}

\section{Introduction}

Vaughan Jones defined in \cite{Jo82} the {\it index} $[M:N]$  of a subfactor $N$ of a II$_1$ factor $M$ as the Murray-von Neumann
dimension of the Hilbert $N$-module $L^2(M)$, obtained by completing $M$
in the norm given by the (unique) trace state $\tau$ on $M$.
He showed in this seminal paper that a subfactor of finite index $N\subset M$ gives rise, in a natural way,
to a whole tower of II$_1$ factors $N \subset M \subset M_1 \subset M_2 \subset \cdots$, all having the same index $[M_{i+1}:M_i]=[M:N]$,
with each $M_{i+1}$ generated by $M_i$ and a projection $e_{i}$ commuting with all elements in $M_{i-1}$ and satisfying $e_iM_ie_i=M_{i-1}e_i$, and
trace determined by $\tau(xe_{i}y)=\lambda \tau(xy)$, $\forall x,y\in M_i$, where $\lambda=[M:N]^{-1}$. In particular, the projections
$\{e_i\}_{i\geq 0}$ satisfy the relations $(a)$ $[e_i, e_j]=0$, if $j\neq i\pm 1$;  $(b)$ $e_ie_{i\pm 1}e_i=\lambda e_i$;
$(c)$ $\tau(we_{i})=\lambda \tau(w)$, $\forall w\in \text{\rm Alg}(1, e_0, ..., e_{i-1})$. This  fact  imposes some striking restrictions  on the index,
$[M:N]=\lambda^{-1}\in \{4\cos^2 \pi/n \mid n\geq 3\}\cup [4, \infty)$, with each one of these values actually occurring (see \cite{Jo82} for all this).

The discovery of this new type of symmetries  had a profound impact on several fields of mathematics.
In particular, it led to {\it subfactor theory}, which aims at understanding the group-like objects generated
by such symmetries, and the way they can ``act'' on factors. There is indeed a group-like structure underlying the Jones tower of algebras: the Hilbert $M$-bimodule (or
{\it correspondence}, in Connes' terminology, see \cite{Co80,Co90}) $L^2(M_1)$ can be viewed as a generalized symmetry of the algebra $M$, with the C$^*$-tensor category it generates
under composition (=~tensor product, or {\it fusion}) giving rise to the $M$-bimodules in the Jones tower, $_M L^2(M_n)_M=(_M L^2(M_1)_M)^{\otimes n}$.
The mathematical object that captures the group-like aspects unraveled
by the dynamics of these symmetries is called the \emph{standard invariant} $\cG_{N,M}$ of $N \subset M$ and it is
defined as the system (or lattice) of inclusions of finite dimensional C$^*$-algebras obtained
by taking the relative commutants in the Jones tower of factors
$$\begin{array}{ccccccccc}
M' \cap M &\subset & M' \cap M_1 &\subset^{e_1} & M' \cap M_2 &\subset^{e_2} & M' \cap M_3 &\subset^{e_3} & \cdots \\
& & \cup & & \cup & & \cup & \\
& & M_1' \cap M_1 & \subset & M_1' \cap M_2 & \subset^{e_2} & M_1' \cap M_3 & \subset^{e_3} & \cdots
\end{array}$$
endowed with the trace $\tau$
inherited from $\cup_n M_n$ and with a representation of the $\lambda$-sequence of projections $e_i, i\geq 1$. The relative commutants appearing in the first row of these inclusions
recover the algebras of endomorphisms of the Hilbert bimodules $_ML^2(M_n)_M$. Their consecutive embeddings are described by a bipartite
graph, denoted $\Gamma_{N,M}$ and called the {\it principal graph} of $N\subset M$. It is a Cayley-type graph that has vertices indexed
by irreducible sub-bimodules $\mathcal H_k \subset$   $_ML^2(M_n)_M$ and describes their fusion. The trace $\tau$ is determined by the dimensions
of these bimodules, $v_k=(\text{\rm dim} (_M{\mathcal H_k}_M))^{1/2}$, which
satisfy the Perron-Frobenius type condition $\Gamma_{N,M}^t\Gamma_{N,M}\vec{v}=\lambda^{-1}\vec{v}$, where $\vec{v}=(v_k)_k$.
In particular, $\|\Gamma_{N,M}\|^2 \leq [M:N]$, with the equality characterizing the {\it amenability} of $\cG_{N,M}$ (a Kesten-type definition). This is
the case if for instance $\Gamma_{N,M}$ is finite, i.e.
when $N\subset M$ has {\it finite depth}.

The objects $\cG_{N,M}$ carry a very rich and subtle algebraic-combinatorial structure, which already for index $[M:N]<4$ led to most surprising
results: of all the Coxeter bipartite graphs $A_n, D_n, E_6, E_7, E_8$ of square norm $<4$, only $A_n, D_{2n}, E_6, E_8$ can occur,
with one $\cG_{N,M}$ for each $A_n, D_{2n}$, two for $E_6$ and two for $E_8$ (\cite{Jo82,Oc88,I91,K91}). This was shown by exploiting
fusion rules obstructions and an
axiomatization of the finite depth standard invariants in \cite{Oc88}.

An important step in understanding $\cG_{N,M}$ was
the axiomatization of arbitrary such objects as $\lambda$-{\it lattices}, in \cite{Po94b}. A $\lambda$-lattice is an abstract system of inclusions
of finite dimensional $C^*$-algebras, $\cG=(A_{ij})_{j\geq i, i=0,1}$, with a trace $\tau$ and a representation of
the Jones projections $\{e_i\}_{i\geq 1}$, satisfying  certain commutation and $\tau$-independence
properties. Due to the relations
in the Jones tower, standard invariants are easily seen to satisfy these axioms and the {\it reconstruction theorem} in \cite{Po94b} associates
in a canonical way to any given such $\lambda$-lattice $\cG$ a (non-hyperfinite) subfactor $N\subset M$ such that $\cG_{N,M}=\cG$, i.e. $A_{ij}=M_i'\cap M_j$, $\forall j\geq i\geq 0$,
where $\{M_i\}_i$ is the Jones tower for $N\subset M$.

If $\lambda^{-1}\in \{4\cos^2 \pi/n \mid n\geq 3\}\cup [4, \infty)$ and
$\{e_i\}_{i\geq 1}$ is a sequence of projections with a trace satisfying the above properties $(a), (b), (c)$ with respect to $\lambda$
(as for instance coming from the Jones tower of factors associated to $N\subset M$ with $[M:N]=\lambda^{-1}$),
then the system of algebras $A_{0,j}=\text{\rm Alg}(1, e_1, ..., e_j)$,
$A_{1j}=\text{\rm Alg}(1, e_2, ..., e_j)$  does check the $\lambda$-lattice axioms, and so there exists a subfactor of index $\lambda^{-1}$ with the relative commutants in its
Jones tower generated by the projections $e_i$ alone.
This was in fact already shown in \cite{Jo82} in the case $\lambda^{-1}< 4$ (corresponding to principal graph equal
to $A_n$, for some $n<\infty$),
and in \cite{Po90} in the case $\lambda^{-1}\geq 4$ (corresponding to principal graph equal to $A_\infty$). We will denote this
$\lambda$-lattice by $\cG^\lambda$ and call it the \emph{Temperley-Lieb-Jones (TLJ) $\lambda$-lattice}.
% the {\it TLJ $\lambda$-lattice} (TLJ stands for Temperley-Lieb-Jones).
Note that $\cG^\lambda$
is contained as a sublattice in any other $\lambda$-lattice $\cG$, thus being in some sense ``minimal'' and conferring it a central role in subfactor theory.

A diagrammatic description of the standard invariants as {\it planar algebras}, was developed by Jones in \cite{Jo99}. Over the last fifteen years,
the  planar algebra formalism grew into a formidable calculus machinery, an extremely efficient framework for concrete computations allowing a plethora of new results. In particular, it made possible the  complete classification of all standard invariants of index $\leq 5$ (see  \cite{JMS13} for a  survey of these results).

Much in parallel to the development of subfactor theory, \emph{quantum groups} were discovered in the context of the quantum inverse scattering method and the quantum Yang-Baxter equation, see \cite{Dr86}. Of particular importance were the $q$-deformations of compact Lie groups \cite{Ji85,Dr86} and how to view them as \emph{topological} quantum groups in the framework of Woronowicz's \emph{compact quantum groups} \cite{Wo86,Wo88,Ro89,Wo95}, a theory initially motivated by Tannaka-Krein and Pontryagin duality for non-abelian groups.

Very remarkably and quite significant for us here, quantum groups, both compact and at roots of unity, but also finitely generated discrete groups, compact Lie groups, rigid C$^*$-tensor categories, etc., can be encoded in the standard invariant of an appropriate subfactor (see \cite{We87,Po92,PW91,Ba98,Xu97,Jo03}). Moreover, a subfactor $N \subset M$ can be viewed as encoding a crossed product type construction of the factor $N$ by the group-like object $\cG_{N,M}$.
The analogy with discrete groups is far reaching. In \cite{Po92,Po99}, it was proved that $\cG_{N,M}$ is a complete invariant for hyperfinite subfactors $N\subset M$
with $\cG_{N,M}$ amenable (in particular for hyperfinite subfactors with finite depth). In other words, an amenable $\cG_{N,M}$
arises from precisely one hyperfinite subfactor, and this should be compared with the fact that amenable groups admit a unique outer action on the hyperfinite II$_1$ factor,
up to cocycle conjugacy \cite{Oc85}.

The main goal of this article is to define the \emph{unitary representation theory} for subfactor related group-like objects (notably $\lambda$-lattices),
to use it to provide a natural framework for rigidity and approximation properties for these objects, such as property~(T) and the Haagerup property,
and to calculate it in the most basic examples. The most natural setting to define this representation theory is for general \emph{rigid C$^*$-tensor categories\footnote{A rigid C$^*$-tensor category is a C$^*$-tensor category that is semisimple, with irreducible tensor unit $\eps \in \cC$ and with every object $\al \in \cC$ having an adjoint $\albar \in \cC$ that is both a left and a right dual of $\al$. For basic definitions and results on rigid C$^*$-tensor categories, we refer to \cite[Sections 2.1 and 2.2]{NT13}.}}, like the category of $M$-bimodules generated by the subfactor $N \subset M$, or the category of finite dimensional representations of a compact group $G$ or a compact quantum group $\bG$.

Thus, if $\cC$ is a rigid C$^*$-tensor category, we consider the \emph{fusion $*$-algebra} $\C[\cC]$ with vector space basis $\Irr(\cC)$ (the irreducible objects in $\cC$) and product given by the fusion rules. We introduce the concept of an \emph{admissible representation} of $\C[\cC]$ and the corresponding notion of positive type function $\vphi : \Irr(\cC) \recht \C$. This allows us to define the \emph{universal C$^*$-algebra} $C_u(\cC)$, as well as property~(T), the Haagerup property and the complete metric approximation property (CMAP) for rigid C$^*$-tensor categories. It is important to note that admissibility of a representation and the positive type of $\vphi$ depend in a subtle way on the tensor category $\cC$ and not only on the fusion rules. Also note that in several examples, $\C[\cC]$ is a $*$-algebra of polynomials and hence, does not have an abstract enveloping C$^*$-algebra.

C$^*$-tensor categories arise in numerous mathematical contexts and play in this way a strong unifying role. Thus,
when $\Gamma$ is a discrete group and $\cC$ is the category of finite dimensional $\Gamma$-graded Hilbert spaces, our representation
theory for $\cC$ coincides with the representation theory of $\Gamma$, while
$C_u(\cC)$ is equal to $C^*(\Gamma)$, the full group C$^*$-algebra of $\Gamma$. When $\cC = \Rep(\bG)$ is the representation category of a compact quantum group $\bG$, we prove that the unitary representations of $\cC$ coincide with the ``central'' representations of the discrete dual $\bGh$ that were considered in a more ad hoc manner in \cite{DFY13}, with $C_u(\cC)$ being isomorphic to a corner of the full C$^*$-algebra of the Drinfel'd double of $\bG$. The Haagerup property and property~(T) for $\cC$ in our sense are then equivalent with their ``central'' counterparts for $\bGh$ defined in \cite{DFY13,Ar14} and shown there to be preserved when passing to a quantum group with the same representation category $\Rep(\bG)$.

When $\cC$ is the bimodule category of a subfactor $N \subset M$, we prove that the representations of $\cC$ correspond to
representations of the {\it symmetric enveloping} ({\it SE}) {\it inclusion} $M \ovt M\op\subset M \boxtimes_{e_N} M\op$
(the quantum double of $N\subset M$) considered in \cite{Po94a,Po99}, i.e. to Hilbert bimodules of
$M \boxtimes_{e_N} M\op$ that are generated by $M \ovt M\op$-central vectors.
This allows us to show that the above mentioned approximation and rigidity properties for the bimodule category $\cC$ of $N\subset M$
are equivalent with the corresponding properties of the
$\lambda$-lattice $\cG=\cG_{N,M}$.
Such properties of $\cG$ were defined before in terms of the \SE-inclusion $M \ovt M\op\subset M \boxtimes_{e_N} M\op$
(see \cite{Po99,Po01,Br14}), following a strategy proposed in \cite{Po94a},  thus having to show each time
that the property is independent of the choice of the subfactor $N\subset M$ with $\cG=\cG_{N,M}$.

An important related object that we consider
is the {\it extended $\lambda$-lattice} (or quantum double) of $\cG$, denoted $\tilde{\cG}$, with an appropriate notion of {\it multipliers} on $\tilde{\cG}$
which we show to correspond to completely positive (more generally completely bounded)
$M \ovt M\op$-bimodular maps on $M \boxtimes_{e_N} M\op$, whenever $N\subset M$ is a subfactor with $\cG_{N,M}=\cG$.

We expect that our representation theory for $\lambda$-lattices will play an important role in better understanding the analytic properties of subfactors and standard invariants, especially of the non-amenable ones. Already in this paper, our unified approach through C$^*$-tensor categories allows
us to bridge from subfactors to quantum groups and take advantage of recent progress made there,
obtaining in this way the following application.

\begin{letterthm}\label{thm.main}
\begin{enumlist}
\item Let $\lambda^{-1} \geq 4$. The representation theory of the TLJ $\lambda$-lattice $\cG^\lambda$ is naturally equivalent with the representation theory of the abelian C$^*$-algebra $C([0,\lambda^{-1}])$, with the regular representation corresponding to left multiplication on $L^2([0,4])$ and with the trivial representation given by evaluation at $\lambda^{-1}$.

\item The $\lambda$-lattices $\cG^\lambda$ have the Haagerup approximation property and the complete metric approximation property.

\item Let $N \subset M$ be a finite index subfactor whose bimodule category is equivalent with $\Rep(\SU_q(n))$ or $\Rep(\PSU_q(n))$ where $n$ is an odd integer greater than or equal to $3$. Then, the standard invariant of $N \subset M$ has property~\pT.
\end{enumlist}
\end{letterthm}

Note that Theorem \ref{thm.main} provides the first subfactors with property~(T) standard invariant and that are not constructed from property~(T) groups, as the subfactors in \cite{BP98}. Theorem \ref{thm.main} also gives the first subfactors whose standard invariant has the Haagerup property without being amenable or constructed from groups with the Haagerup property.

In Section \ref{sec.permanence}, we prove a number of results on the permanence under various constructions of the Haagerup property, weak amenability, property~(T),~.... As an application, we show that the Fuss-Catalan $\lambda$-lattices of \cite{BJ95} have the Haagerup property and CMAP. In Section \ref{sec.approx-rigid-prop}, we give several equivalent definitions of property~(T), one being the existence of a projection in $C_u(\cC)$ onto the trivial representation, and use this in Section \ref{sec.permanence} to deduce that property~(T) passes to quotients. This fact
was proved in \cite{Po99} for $\lambda$-lattices by using subtle analytic arguments, while in our present framework the proof becomes completely algebraic and straightforward.

Theorem \ref{thm.main} is proved as Theorems \ref{thm.main-temperley-lieb} and \ref{thm.main-SUq3} below. To do this, we combine our representation framework with the beautiful recent work in \cite{DFY13,Ar14} on the compact quantum groups $\SU_q(n)$ of \cite{Wo86,Wo88}. More precisely, we use the main result of \cite{DFY13} saying that the dual of $\SU_q(2)$ has the central Haagerup property and the central CMAP, with all central states on the C$^*$-algebra of $\SU_q(2)$ being determined, and we use \cite{Ar14}, where the central property~(T) is established for the dual of $\SU_q(n)$ ($0 < q < 1$ and $n \geq 3$ an odd integer).

\section{Representation theory for subfactors}\label{sec.rep-theory}

Let $N \subset M$ be an extremal\footnote{Recall that a finite index subfactor $N \subset M$ is called extremal if the Jones projection $e_N \in M_1 = \langle M,e_N\rangle$ satisfies $E_{M' \cap M_1}(e_N) = [M:N]^{-1} 1$, see \cite[1.2.5]{Po92}. Also note that all irreducible subfactors are extremal.} subfactor. We denote by $M \boxt_{e_N} M\op$ the \emph{symmetric enveloping algebra} in the sense of \cite{Po94a,Po99}. We call $M \ovt M\op \subset M \boxt_{e_N} M\op$ the \emph{\SE-inclusion} associated with the extremal subfactor $N \subset M$. We often denote $T := M \ovt M\op$ and $S = M \boxt_{e_N} M\op$.

\begin{definition}\label{def.rep-subf}
Let $N \subset M$ be an extremal subfactor with associated \SE-inclusion $T \subset S$. An \emph{\SE-correspondence} of $N \subset M$ is a Hilbert $S$-bimodule $\bim{S}{\cH}{S}$ that is generated by $T$-central vectors.
\end{definition}

More precisely, to every $S$-bimodule $\bim{S}{\cH}{S}$, we associate the space of $T$-central vectors $\cH_T := \{\xi \in \cH \mid \forall x \in T : x \xi = \xi x\}$. We call $\bim{S}{\cH}{S}$ an \SE-correspondence of $N \subset M$ if the linear span of $S \cH_T S$ is dense in $\cH$.

The \emph{trivial \SE-correspondence} of $N \subset M$ is given by the $S$-bimodule $\bim{S}{L^2(S)}{S}$, while the \emph{coarse \SE-correspondence} of $N \subset M$ is given by the $S$-bimodule $\bim{S}{L^2(S) \ot_T L^2(S)}{S}$.

Let $\xi_0 \in \cH_T$ be a unit vector. Since $\xi_0$ is $T$-central, the state $x \mapsto \langle x \xi_0,\xi_0 \rangle$ on $S$ is $T$-central. By \cite[Proposition 2.6]{Po99}, the inclusion $T \subset S$ is irreducible and we conclude that $\langle x \xi_0,\xi_0 \rangle = \tau(x) = \langle \xi_0 x , \xi_0 \rangle$ for all $x \in S$. So, there is a unique normal, completely positive, unital, trace preserving, $T$-bimodular map $\psi : S \recht S$ satisfying
\begin{equation}\label{eq.link}
\langle x \xi_0 y ,\xi_0 \rangle = \tau(x \psi(y)) \quad\text{for all}\;\; x,y \in S \; .
\end{equation}
Conversely, every normal, $T$-bimodular, completely positive map $\psi : S \recht S$ gives rise to an \SE-correspondence $\bim{S}{\cH}{S}$ of the subfactor $N \subset M$ and a $T$-central vector $\xi_0 \in \cH_T$ such that \eqref{eq.link} holds. So, these $T$-bimodular completely positive maps play the role of functions of positive type on a group. In a similar way, we have the analogue of completely bounded multipliers.

\begin{definition}\label{def.mult-subf}
Let $N \subset M$ be an extremal subfactor with associated \SE-inclusion $T \subset S$. We call \emph{\SE-multiplier} of $N \subset M$ every normal $T$-bimodular linear map $\psi : S \recht S$. When moreover $\psi$ is completely positive, resp.\ completely bounded, we call $\psi$ a \emph{cp \SE-multiplier}, resp.\ \emph{cb \SE-multiplier} of $N \subset M$.
\end{definition}

We call $\bim{S}{\cH}{S}$ a \emph{cyclic \SE-correspondence} if there exists a single vector $\xi_0 \in \cH_T$ such that the linear span of $S \xi_0 S$ is dense in $\cH$. By the discussion above, we see that there is a natural correspondence between cyclic \SE-correspondences of $N \subset M$ and cp \SE-multipliers of $N \subset M$.

It turns out that the \SE-correspondences of a subfactor can be exactly described by the representations of an associated universal C$^*$-algebra that we construct now.

Given a finite index subfactor $N \subset M$ with Jones tower $N \subset M \subset M_1 \subset M_2 \subset \cdots$, we consider the C$^*$-tensor category $\cC$ of all $M$-bimodules that are isomorphic to a finite direct sum of $M$-subbimodules of $\bim{M}{L^2(M_n)}{M}$ for some $n$. We denote by $\Irr(\cC)$ the set of equivalence classes of irreducible $M$-bimodules in $\cC$. We define $\C[\cC]$ to be the fusion $*$-algebra of $\cC$~: the free vector space with basis $\Irr(\cC)$, $*$-operation given by taking the adjoint bimodule and product given by the fusion rules.

Assume now that $N \subset M$ is extremal and consider the associated \SE-inclusion $T \subset S$. By \cite[Theorem 4.5]{Po99}, we can uniquely decompose the $T$-bimodule $L^2(S)$ into a direct sum of irreducible $T$-subbimodules $(\cL_\pi)_{\pi \in \Irr(\cC)}$ labeled by the elements of $\Irr(\cC)$ such that $\cL_\pi \cong \pi \ot \pibar\op$ as $M \ovt M\op$-bimodules. Every $T$-bimodule $\cL_\pi$ appears with multiplicity $1$ in $L^2(S)$. Because $N \subset M$ is extremal, all bimodules $\pi \in \cC$ have equal left and right $M$-dimension that we denote as $d(\pi)$.

\begin{theorem}\label{thm.full-Cstar-subfactor}
Let $N \subset M$ be an extremal subfactor with associated \SE-inclusion $T \subset S$ and category of $M$-bimodules $\cC$.
For every \SE-correspondence $\bim{S}{\cH}{S}$ of $N \subset M$, the linear map uniquely defined by
\begin{equation}\label{eq.def-Theta}
\begin{split}
\Theta : \C[\cC] \recht B(\cH_T) : \; & \Theta(\pi)(\xi) = \frac{1}{d(\pi)} \sum_i m_i \xi m_i^* \quad\text{for all}\;\; \pi \in \Irr(\cC) \\
& \text{where}\;\; m_i \in S \;\text{is an orthonormal basis of $\cL_\pi$ as a right $T$-module,}
\end{split}
\end{equation}
is a unital $*$-representation of $\C[\cC]$ on the Hilbert space $\cH_T$ of $T$-central vectors and satisfies $\|\Theta(\al)\| \leq d(\al)$ for all $\al \in \cC$.

Let $\bim{S}{\cH}{S}$ and $\bim{S}{\cH'}{S}$ be two \SE-correspondences of $N \subset M$ with associated $*$-representations $\Theta$ and $\Theta'$ of $\C[\cC]$. The map from $\Mor(\bim{S}{\cH}{S},\bim{S}{\cH'}{S})$ to $\Mor(\Theta,\Theta')$ given by restricting an $S$-bimodular bounded operator $V : \cH' \recht \cH$ to the subspace $\cH'_T \subset \cH'$ is a bijective, isometric map.

In particular, $\bim{S}{\cH}{S}$ and $\bim{S}{\cH'}{S}$ are unitarily conjugate as $S$-bimodules if and only if the $*$-representations $\Theta$ and $\Theta'$ are unitarily equivalent.
\end{theorem}

Before proving Theorem \ref{thm.full-Cstar-subfactor}, we define the universal C$^*$-algebra of a subfactor $N \subset M$ and make a few remarks. Since $\|\Theta(\pi)\| \leq d(\pi)$ for all $\pi \in \Irr(\cC)$, the following definition makes sense.

\begin{definition}\label{def.full-Cstar-subfactor}
Let $N \subset M$ be an extremal subfactor with associated \SE-inclusion $T \subset S$ and category of $M$-bimodules $\cC$.
We call a $*$-representation of $\C[\cC]$ on a Hilbert space \emph{admissible} if it is unitarily conjugate to the $*$-representation associated with an \SE-correspondence $\bim{S}{\cH}{S}$ of $N \subset M$ as in Theorem \ref{thm.full-Cstar-subfactor}.

We define the \emph{universal C$^*$-algebra} $C_u(N \subset M)$ as the completion of $\C[\cC]$ with respect to a universal admissible representation of $\C[\cC]$.
\end{definition}

In Section \ref{sec.Cstar-algebra-tensor-cat}, we will see that $C_u(N \subset M)$ can be defined intrinsically in terms of the C$^*$-tensor category $\cC$.

\begin{remark}\label{rem.fell-topology-SE}
It is now straightforward to also define the Fell topology on \SE-correspondences in such a way that it coincides with the usual Fell topology on the representations of the C$^*$-algebra $C_u(N \subset M)$. More concretely, for all $T$-central unit vectors $\xi \in \cH$, we consider the (unital, trace preserving) normal cp $T$-bimodular map $\psi_\xi : S \recht S$ given by $\tau(x\psi_\xi(y)) = \langle x \xi y,\xi\rangle$. We call these maps $\psi_\xi$ the \emph{coefficients} of $\bim{S}{\cH}{S}$. Given two \SE-correspondences $\bim{S}{\cH}{S}$ and $\bim{S}{\cH'}{S}$, we say that $\cH$ is \emph{weakly contained} in $\cH'$ if every coefficient of $\cH$ can be approximated, in the pointwise $\|\,\cdot\,\|_2$-topology, by convex combinations of coefficients of $\cH'$.

Denote by $\Theta,\Theta'$ the $*$-representations associated with $\cH,\cH'$ as in Theorem \ref{thm.full-Cstar-subfactor}. One checks as follows that $\cH$ is weakly contained in $\cH'$ if and only if $\Theta$ is weakly contained in $\Theta'$, i.e.\ $\|\Theta(x)\| \leq \|\Theta'(x)\|$ for all $x \in \C[\cC]$. If $\xi \in \cH$, $\xi_i \in \cH'$ are $T$-central vectors and $\lambda_i \in [0,1]$ with $\sum_i \lambda_i = 1$, it follows from Lemma \ref{lem.projection-central-vectors} below that
$$\bigl\|\psi_{\xi}(y) - \sum_i \lambda_i \psi_{\xi_i}(y)\bigr\|_2^2 = \sum_{\al \in \Irr(\cC)} \|y_{\albar}\|_2^2 \, \frac{1}{d(\al)^2} \, \bigl|\langle \Theta(\al) \xi,\xi \rangle - \sum_i \lambda_i \langle \Theta'(\al) \xi_i , \xi_i \rangle\bigr|^2 \; ,$$
where we use the decomposition $y = \sum_\al y_\al$ with $y_\al \in \cL_\al$. We conclude that $\cH$ is weakly contained in $\cH'$ if and only if all coefficients of $\Theta$ (in the usual sense) can be pointwise approximated by convex combinations of coefficients of $\Theta'$.
\end{remark}

\begin{remark}\label{rem.not-all-are-admissible}
For a diagonal (locally trivial) subfactor $N \subset M$ associated with an outer action of a countable group $G \actson M$ and a finite generating set of $G$ (see
\cite[Section 5.1.5]{Po92}), by \cite[Theorem 3.3]{Po99}, we can identify $S \cong T \rtimes G$ where $G \actson T = M \ovt M\op$ is the diagonal action. We can further identify $\Irr(\cC) \cong G$ and $\C[\cC] \cong \C[G]$. Clearly, every $*$-representation of $\C[G]$ is admissible and $C_u(N \subset M) \cong C^*(G)$, the full C$^*$-algebra of $G$.

In general however, not all $*$-representations of $\C[\cC]$ need be admissible. In particular, when $N \subset M$ is an $A_\infty$ subfactor with index $\lambda^{-1} = [M:N] \geq 4$, we can identify $\C[\cC]$ with the polynomial $*$-algebra $\C[X]$ with $X^* = X$. The $M$-bimodule $L^2(M_1)$ corresponds to the monomial $X$. For every $t \in \R$, we consider the one-dimensional $*$-representation $\counit_t : \C[X] \recht \C : \counit_t(P) = P(t)$. If $\counit_t$ is admissible, we must have
$$|t| = |\counit_t(L^2(M_1))| \leq d(L^2(M_1)) = \lambda^{-1} \; .$$
So already here, we get that $\counit_t$ is not admissible for $|t| > \lambda^{-1}$. In Section \ref{sec.rep-theory-TLJ}, we will see that $\counit_t$ is admissible if and only if $t \in [0,\lambda^{-1}]$, so that $(\counit_t)_{t \in [0,\lambda^{-1}]}$ is the complete list of all irreducible admissible representations of $\C[\cC]$ and $C_u(N \subset M) \cong C([0,\lambda^{-1}])$. We also refer to Section \ref{sec.different-Cstar-algebras} for a more complete discussion.
\end{remark}

\begin{remark}\label{rem.nonextremal}
Let $N \subset M$ be a finite index subfactor that is not necessarily extremal. As above, consider the associated C$^*$-tensor category $\cC$ of $M$-bimodules. For every $\al \in \cC$, we denote by $d_l(\al)$ the dimension of $\al$ as a left $M$-module and by $d_r(\al)$ the dimension of $\al$ as a right $M$-module.

We can still define an \SE-inclusion $T \subset S$ with $T \cong M \ovt M\op$, but $S$ is no longer tracial. We will rather have a natural almost periodic normal faithful state $\om$ on $S$ whose modular group is described in terms of the ratios $d_l(\al)/d_r(\al)$, $\al \in \Irr(\cC)$. It is possible to repeat the construction of \cite{Po94a,Po99} and define $S$ through a universal property, generated by $M \ovt M\op$ and a projection that is the Jones projection for both $N \subset M$ and $N\op \subset M\op$. Instead, since this is useful for us later anyway, we sketch how to adapt the categorical construction of \cite{LR94,Ma99} to the non-extremal setting.

This construction actually makes sense for an arbitrary full C$^*$-tensor subcategory $\cC$ of the category of finite index $M$-bimodules. We choose a set of representatives $\cH_\al$ for all $\al \in \Irr(\cC)$. We fix anti-unitary operators $j_\al : \cH_\al \recht \cH_{\albar}$ for all $\al \in \Irr(\cC)$ satisfying $j_\al(x \xi y) = y^* j_\al(\xi) x^*$ for all $x,y \in M$, $\xi \in \cH_\al$. Note that $j_\al$ is uniquely determined up to multiplication by a scalar of modulus one. For all $\eta \in \cC$, $\gamma \in \Irr(\cC)$ and $V,W \in \Mor(\eta,\gamma)$, we have that $W^* V \in \End(\gamma) = \C1$. In this way, we can view $\Mor(\eta,\gamma)$ as a finite-dimensional Hilbert space with scalar product $W^* V = \langle V,W\rangle \, 1$. Denote by $\cH^0_\al \subset \cH_\al$ the subspace of $M$-bounded vectors. Define the vector space $S_0$ as the algebraic direct sum
$$S_0 = \bigoplus_{\al \in \Irr(\cC)} \bigl(\cH^0_\al \otalg \overline{\cH^0_\al}\bigr) \; .$$
Denote by $\delta_\al : \cH^0_\al \otalg \overline{\cH^0_\al} \recht S_0$ the embedding as the $\al$'th direct summand. Using an orthonormal basis (onb) of the Hilbert space $\Mor(\al \ot \be,\gamma)$, we define a multiplication on $S_0$ by
$$\delta_\al(\xi \ot \mubar) \; \delta_\be(\xi' \ot \mupbar) = \sum_{\gamma \in \Irr(\cC), i, V_i \; \text{onb} \; \Mor(\al \ot \be,\gamma)} \quad \delta_\gamma\bigl( V_i^* (\xi \ot_M \xi') \; \ot \; \overline{ V_i^* (\mu \ot_M \mu')} \bigr) \; . $$
Note that we sum only finitely many terms, since only those $\gamma \in \Irr(\cC)$ that arise as a submodule of $\al \ot \be$ contribute to the sum. Also note that the definition of the product is independent of the choice of bases $V_i$. In this way, one gets the following formula, implying the associativity of the product:
\begin{multline*}\delta_\al(\xi \ot \mubar) \; \delta_\be(\xi' \ot \mupbar) \; \delta_\gamma(\xi'' \ot \muppbar) \\ = \sum_{\eta \in \Irr(\cC), i, V_i \; \text{onb} \; \Mor(\al \ot \be \ot \gamma,\eta)} \;\; \delta_\eta\bigl( V_i^* (\xi \ot_M \xi' \ot_M \xi'') \; \ot \; \overline{ V_i^* (\mu \ot_M \mu' \ot_M \mu'')} \bigr) \; .
\end{multline*}
We also define a $*$-operation on $S_0$ given by $\delta_\al(\xi \ot \mubar)^* = \delta_{\albar}(j_\al(\xi) \ot \overline{j_\al(\mu)})$.
Since $j_\al$ is uniquely determined up to a scalar of modulus one, also the $*$-operation is independent of all choices and turns $S_0$ into a unital $*$-algebra. The map $a \ot b\op \mapsto \delta_\eps(a \ot \overline{b^*})$ identifies $M \otalg M\op$ with a $*$-subalgebra of $S_0$.

Finally define the linear functional $\om : S_0 \recht \C$ given by
$$\om(\delta_\al(\xi \ot \mubar)) = 0 \quad\text{if}\;\; \al \neq \eps \quad\text{and}\quad \om(a \ot \bbar) = \tau(a) \, \overline{\tau(b)} \quad\text{for all}\;\; a,b \in M \; .$$
A direct computation yields, for all $x,y \in S_0$,
\begin{align*}
\om(y^* x) & = \sum_{\al \in \Irr(\cC)} d_r(\al)^{-1} \langle x_\al,y_\al \rangle \quad\text{and}\\
\om(x y^*) & = \sum_{\al \in \Irr(\cC)} d_l(\al)^{-1} \langle x_\al,y_\al \rangle \; .
\end{align*}
Denote by $\cK$ the Hilbert space completion of $S_0$ with respect to the scalar product $\langle x,y\rangle = \om(y^*x)$. Using that the vectors in $\cH^0_\al$ are bounded, one checks that both the left and right multiplication by elements in $S_0$ extend to bounded operators on $\cK$. We define $S$ as the von Neumann algebra generated by left multiplication operators on $\cK$. By construction, $\cK = \bigoplus_{\al \in \Irr(\cC)} (\cH_\al \ot \overline{\cH_\al})$. Therefore, the embedding of $M \otalg M\op$ into $S_0$ extends to a normal embedding of $T := M \ovt M\op$ into $S$. The functional $\om$ on $S_0$ extends to a normal faithful state on $S$ whose restriction to $T$ equals the trace and whose modular automorphism group satisfies
$$\sigma_t^\vphi(\delta_\al(\xi \ot \mubar)) = \Bigl(\frac{d_r(\al)}{d_l(\al)}\Bigr)^{it} \; \delta_\al(\xi \ot \mubar)$$
for all $\al \in \Irr(\cC)$ and $\xi,\mu \in \cH^0_\al$.

Note that by construction $T' \cap S = \C 1$ and as a $T$-bimodule, $L^2(S)$ the direct sum of the $T$-bimodules $\cH_\al \ot \overline{\cH_\al}$, each appearing with multiplicity one.

So far, the construction of the \SE-inclusion $T \subset S$ made sense for an arbitrary full C$^*$-tensor subcategory $\cC$ of the category of finite index $M$-bimodules. Now assume that $\cC$ is the category generated by a finite index subfactor $N \subset M$. Putting $M_{-1} := N$, $M_0 := M$ and choosing a tunnel $\cdots \subset M_{-2} \subset M_{-1} \subset M_0$, we can identify in the following way
$$(1 \ot M_{-n}\op)' \cap S \cong M_n \quad\text{and}\quad (M_{-n} \ot 1)' \cap S \cong M_n\op \; .$$
To every bounded vector $\xi \in \cH^0_\al$ are associated the bounded linear operators
$$L_\xi : L^2(M) \recht \cH_\al : L_\xi(x) = \xi x \quad\text{and}\quad R_\xi : L^2(M) \recht \cH_\al : R_\xi(x) = x \xi \; .$$
It is then straightforward to check that
$$\Gamma : S_0 \recht B(L^2(M)) : \Gamma(\delta_\al(\xi \ot \mubar)) = R_\mu^* L_\xi$$
is a unital $*$-homomorphism. Of course, $\Gamma$ is by no means normal on $S$. But, the restriction of $\Gamma$ to $(1 \ot M_{-n}\op)' \cap S_0 = (1 \ot M_{-n}\op)' \cap S$ is a normal $*$-isomorphism of $(1 \ot M_{-n}\op)' \cap S$ onto $M_n$, the latter being realized as the commutant of the right $M_{-n}$-action on $L^2(M)$.

We continue to view $M_n$ as the commutant of the right $M_{-n}$ action on $L^2(M)$. Denote by $\tau$ the unique normalized trace on $M_n$. Then, there is a unique positive invertible operator $Q_n \in \cZ(M_{-n}' \cap M)$ such that $\tau(Ja^*J) = \tau(a Q_n)$ for all $a \in M_{-n}' \cap M$. One of the equivalent characterizations of extremality amounts to all $Q_n$ being equal to $1$. In general, we define the normal faithful state $\om_n$ on $M_n$ by the formula $\om_n(x) = \tau(a JQ_n^{-1} J)$. We then have
$$\om_n(\Gamma(x)) = \om(x) \quad\text{for all}\;\; x \in (1 \ot M_{-n}\op)' \cap S \; .$$

It is possible to define \SE-correspondences of an arbitrary finite index subfactor in terms of $S$-bimodules generated by $T$-central vectors, and to define the universal C$^*$-algebra $C_u(N \subset M)$. We do not go further into this since in Section \ref{sec.mult-tensor-cat}, we will define these concepts in even greater generality, for abstract rigid C$^*$-tensor categories.
\end{remark}

We now prove Theorem \ref{thm.full-Cstar-subfactor}. As above, we decompose $L^2(S)$ into the direct sum of the irreducible $T$-bimodules $(\cL_\pi)_{\pi \in \Irr(\cC)}$. We denote by $\cL^0_\pi \subset \cL_\pi$ the space of $T$-bounded vectors. Since $T \subset S$ is irreducible, we have that $\cL^0_\pi = \cL_\pi \cap S$.

\begin{lemma}\label{lem.projection-central-vectors}
Let $N \subset M$ be an extremal subfactor with associated \SE-inclusion $T \subset S$. Let $\bim{S}{\cH}{S}$ be an \SE-correspondence of $N \subset M$ and denote by $\cH_T \subset \cH$ the subspace of $T$-central vectors. Denote by $p_T : \cH \recht \cH_T$ the orthogonal projection of $\cH$ onto $\cH_T$. Using the notation of \eqref{eq.def-Theta}, we get
$$p_T(x \xi y^*) = \frac{\tau(xy^*)}{d(\al)} \; \Theta(\al) \xi \quad\text{for all}\;\; \al,\be \in \Irr(\cC) \; , \;\; x \in \cL^0_\al \; ,  \;\; y \in \cL^0_\be \; , \;\; \xi \in \cH_T \; .$$
In particular, $p_T(x \xi y^*) = 0$ when $\al \neq \be$.
\end{lemma}
\begin{proof}
Fix $\xi \in \cH_T$ and $\al,\be \in \Irr(\cC)$. We claim that the linear map
$$V : \cL^0_\al \ot_T \overline{\cL^0_\be} \recht \cH : x \ot y^* \mapsto x \xi y^*$$
extends to a $T$-bimodular bounded operator $V : \cL_\al \ot_T \overline{\cL_\be} \recht \cH$. We may assume that $\|\xi\| = 1$. Then, $x \mapsto \langle x \xi,\xi\rangle$ is a normal $T$-central state on $S$. Since $T' \cap S = \C 1$, it is equal to the trace. Therefore, there is a unique isometry $W : \cL_\al \recht \cH$ satisfying $W(x) = x \xi$ for all $x \in \cL^0_\al$. Take an orthonormal basis $(v_i)_{i=1,\ldots,n}$ of $\cL_\be$ as a right $T$-module. Note that $v_i \in \cL^0_\be$. For every $i$, we have the bounded operator $L_i : \cL_\al \recht \cL_\al \ot_T \overline{\cL_\be} : L_i(x) = x \ot_T v_i^*$. Using that $\xi$ is $T$-central, we find that
$$V(\mu) = \sum_{i=1}^n \bigl(W(L_i^*(\mu))\bigr) v_i^* \quad\text{for all}\;\; \mu \in \cL^0_\al \ot_T \overline{\cL^0_\be} \; ,$$
so that indeed, $V$ is bounded.

For every $\eta \in \cH_T$, we define the bounded $T$-bimodular map $V_\eta : L^2(T) \recht \cH : V_\eta(x) = x \eta$ for all $x \in T$. So, for every $\eta \in \cH_T$, we get that
$$V_\eta^* V : \cL_\al \ot_T \overline{\cL_\be} \recht L^2(T)$$
is a bounded $T$-bimodular operator with $\|V_\eta^* V\| \leq \|\eta\| \, \|V\|$. When $\al \neq \be$, $\Mor(\al \ot \bebar,\eps) = \{0\}$ and we conclude that $V_\eta^* V = 0$ for all $\eta \in \cH_T$. This implies that $p_T \circ V = 0$, meaning that $p_T(x \xi y^*) = 0$ for all $x \in \cL^0_\al$, $y \in \cL^0_\be$.

When $\al = \be$, we find that all $V_\eta^* V$ are a scalar multiple of the canonical $T$-bimodular map $s_\al^* : \cL_\al \ot_T \overline{\cL_\be} \recht L^2(T)$. By Riesz' theorem, we find $\xi_0 \in \cH_T$ such that $V_\eta^* V = \langle \xi_0,\eta \rangle \; s_\al^*$ for all $\eta \in \cH_T$. This formula says that
\begin{equation}\label{eq.eenbegin}
\langle  x \xi y^* , \eta \rangle = \tau(xy^*) \; \langle \xi_0,\eta \rangle \quad\text{for all}\;\; x, y \in \cL^0_\al \; , \; \eta \in \cH_T \; .
\end{equation}
Let $(w_j)$ be an orthonormal basis of $\cL_\al$ as a right $T$-module. We get that
$$\langle \Theta(\al) \xi, \eta \rangle = \frac{1}{d(\al)} \sum_j \langle w_j \xi w_j^* , \eta \rangle = \sum_j \frac{\tau(w_j w_j^*)}{d(\al)} \; \langle \xi_0,\eta \rangle = d(\al) \; \langle \xi_0,\eta \rangle$$
for all $\eta \in \cH_T$, where we used that the dimension of $\cL_\al$ as a $T$-module equals $d(\al)^2$. This means that $\xi_0 = d(\al)^{-1} \Theta(\al) \xi$. Then \eqref{eq.eenbegin} becomes
$$p_T(x\xi y^*) = \frac{\tau(xy^*)}{d(\al)} \; \Theta(\al) \xi \quad\text{for all}\;\; x,y \in \cL^0_\al \; .$$
\end{proof}

\begin{proof}[Proof of Theorem \ref{thm.full-Cstar-subfactor}]
Let $\al \in \Irr(\cC)$ and choose an orthonormal basis $v_i \in \cL^0_\al$ of $\cL_\al$ as a right $T$-module. Using the trace preserving conditional expectation $E : S \recht T$, we get for all $x \in T$ and $\xi \in \cH_T$ that
$$x \sum_i v_i \xi v_i^* = \sum_{i,j} v_j E(v_j^* x v_i) \xi v_i^* = \sum_{i,j} v_j \xi E(v_j^* x v_i) v_i^* = \sum_j v_j \xi v_j^* \; x \; .$$
So, $\Theta(\al) \xi \in \cH_T$. Also, for all $\xi,\eta \in \cH_T$, we have that $\|y \xi\| = \|y\|_2 \, \|\xi\|$ and $\|\eta y \| = \|y\|_2 \, \|\eta\|$ for all $y \in S$. Therefore, using that the dimension of $\cL_\al$ as a $T$-module is $d(\al)^2$, we get that
\begin{align*}
|\langle \Theta(\al) \xi ,\eta \rangle|^2 & \leq \frac{1}{d(\al)^2} \Bigl( \sum_i |\langle v_i \xi , \eta v_i \rangle| \Bigr)^2
 \leq \frac{1}{d(\al)^2} \Bigl( \sum_i \|v_i \xi\|^2 \Bigr) \; \Bigl( \sum_i \|\eta v_i \|^2 \Bigr) \\
& = \frac{1}{d(\al)^2} \Bigl( \sum_i \|v_i\|_2^2 \Bigr)^2 \; \|\xi\|^2 \; \|\eta\|^2
 = d(\al)^2 \; \|\xi\|^2 \; \|\eta\|^2 \; .
\end{align*}
We conclude that $\Theta(\al)$ is a bounded operator on $\cH_T$ with $\|\Theta(\al)\| \leq d(\al)$.

Fix a nonzero element $x \in \cL^0_\al$. Using Lemma \ref{lem.projection-central-vectors}, we get for all $\xi,\eta \in \cH_T$ that
\begin{align*}
\|x\|_2^2 \; \langle \Theta(\al) \xi, \eta \rangle &= d(\al) \; \langle p_T(x \xi x^*) , \eta \rangle
= d(\al) \; \langle x \xi x^* , \eta \rangle \\
&= d(\albar) \; \langle \xi , p_T(x^* \eta x) \rangle
= \|x\|_2^2 \; \langle \xi , \Theta(\albar) \eta \rangle \; .
\end{align*}
So, $\Theta(\al)^* = \Theta(\albar)$.

Finally, fix $\al,\be \in \Irr(\cC)$. Denote by $\cH^0_\al \subset \cH_\al$ and $\cH^0_\be \subset \cH_\be$ the subspaces of $M$-bounded vectors. Choose orthonormal bases $a_i \in \cH^0_\al$ and $b_j \in \cH^0_\be$ for $\cH_\al$ and $\cH_\be$ as right $M$-modules. Using the notation of Remark \ref{rem.nonextremal}, the elements $\sqrt{d(\al)} \delta_\al(a_{i_1} \ot \overline{a_{i_2}})$ form an orthonormal basis of $\cL_\al$ as a right $T$-module. Similarly, $\sqrt{d(\be)}\delta_\be(b_{j_1} \ot \overline{b_{j_2}})$ is an orthonormal basis of $\cL_\be$. We conclude that for all $\xi \in \cH_T$,
$$\Theta(\al)(\Theta(\be)\xi) = \sum_{i_1,i_2,j_1,j_2} \; \delta_\al(a_{i_1} \ot \overline{a_{i_2}}) \; \delta_\be(b_{j_1} \ot \overline{b_{j_2}}) \; \xi  \; \delta_\be(b_{j_1} \ot \overline{b_{j_2}})^* \; \delta_\al(a_{i_1} \ot \overline{a_{i_2}})^* \; .$$
For every $\gamma \in \Irr(\cC)$ that appears as an $M$-subbimodule of $\al \ot \be$, we fix an orthonormal basis $V_{\gamma,k}$ for $\Mor(\al \ot \be,\gamma)$ with $k=1,\ldots,\mult(\gamma,\al \ot \be)$. Using the formula for the product in $S$ as explained in Remark \ref{rem.nonextremal}, we get that
\begin{multline}\label{eq.thetaalbe}
\Theta(\al)(\Theta(\be)\xi) = \sum_{\gamma_1,\gamma_2 \in \Irr(\cC),k_1,k_2} \; \sum_{i_1,i_2,j_1,j_2} \; \delta_{\gamma_1}\bigl( V_{\gamma_1,k_1}^* (a_{i_1} \ot_M b_{j_1}) \ot
\overline{V_{\gamma_1,k_1}^* (a_{i_2} \ot_M b_{j_2})}\bigr) \; \xi \\ \delta_{\gamma_2}\bigl( V_{\gamma_2,k_2}^* (a_{i_1} \ot_M b_{j_1}) \ot
\overline{V_{\gamma_2,k_2}^* (a_{i_2} \ot_M b_{j_2})}\bigr)^* \; .
\end{multline}
Note that $(a_i \ot_M b_j)_{i,j}$ is an orthonormal basis of $\cH_\al \ot_M \cH_\be$ as a right $M$-module. Fix orthonormal bases $(c_{\gamma,l})_l$ for $\cH_\gamma$ as a right $M$-module. Then also $d_{\gamma,k,l} := V_{\gamma,k}(c_{\gamma,l})$ indexed by $\gamma,k,l$ is an orthonormal basis of $\cH_\al \ot_M \cH_\be$. Since $V_{\gamma,k}^*$ is $M$-bimodular and $\xi$ is $T$-central, we can make a change of basis and get, for every fixed $\gamma_1,\gamma_2 \in \Irr(\cC)$ and $k_1,k_2$ that
{\allowdisplaybreaks
\begin{align*}
& \sum_{i_1,i_2,j_1,j_2} \; \delta_{\gamma_1}\bigl( V_{\gamma_1,k_1}^* (a_{i_1} \ot_M b_{j_1}) \ot
\overline{V_{\gamma_1,k_1}^* (a_{i_2} \ot_M b_{j_2})}\bigr) \; \xi \\*[-2ex] & \hspace{6cm} \delta_{\gamma_2}\bigl( V_{\gamma_2,k_2}^* (a_{i_1} \ot_M b_{j_1}) \ot
\overline{V_{\gamma_2,k_2}^* (a_{i_2} \ot_M b_{j_2})}\bigr)^* \\[2ex]
= & \sum_{\gamma_3,\gamma_4 \in \Irr(\cC), k_3,k_4,l_1,l_2} \; \delta_{\gamma_1}\bigl( V_{\gamma_1,k_1}^*(d_{\gamma_3,k_3,l_1}) \ot \overline{V_{\gamma_1,k_1}^* (d_{\gamma_4,k_4,l_2})}\bigr) \; \xi
\\*[-2ex] & \hspace{6cm} \delta_{\gamma_2}\bigl(V_{\gamma_2,k_2}^*(d_{\gamma_3,k_3,l_1}) \ot \overline{V_{\gamma_2,k_2}^* (d_{\gamma_4,k_4,l_2})}\bigr)^* \\[2ex]
= & \; \delta_{\gamma_1,\gamma_2} \, \delta_{k_1,k_2} \; \sum_{l_1,l_2} \delta_{\gamma_1}(c_{\gamma_1,l_1} \ot \overline{c_{\gamma_1,l_2}}) \; \xi \; \delta_{\gamma_1}(c_{\gamma_1,l_1} \ot \overline{c_{\gamma_1,l_2}})^* \\[1ex]
= & \; \delta_{\gamma_1,\gamma_2} \, \delta_{k_1,k_2} \, \Theta(\gamma_1) \xi \; .
\end{align*}}
In combination with \eqref{eq.thetaalbe} and using that we sum over $\mult(\gamma,\al \ot \be)$ indices $k$, we get that
$$\Theta(\al) (\Theta(\be) \xi) = \sum_{\gamma \in \Irr(\cC)} \, \mult(\gamma,\al \ot \be) \; \Theta(\gamma) \xi = \Theta(\al\be) \xi \; .$$
So we have proved that $\Theta$ is a $*$-representation of $\C[\cC]$ on $\cH_T$.

It remains to prove that if $\bim{S}{\cH}{S}$ and $\bim{S}{\cH'}{S}$ are \SE-correspondences of $N \subset M$ with associated $*$-representations $\Theta$ and $\Theta'$ given by \eqref{eq.def-Theta}, then the map $\Psi : \Mor(\bim{S}{\cH}{S},\bim{S}{\cH'}{S}) \recht \Mor(\Theta,\Theta')$ given by restricting an $S$-bimodular bounded operator $V : \cH' \recht \cH$ to the subspace $\cH'_T \subset \cH'$ of $T$-central vectors is bijective and isometric.

Using the direct sum of $\cH$ and $\cH'$, we may assume that $\cH=\cH'$ and we have to prove that $\Psi$ is an isomorphism of the von Neumann algebra $\End(\bim{S}{\cH}{S})$ onto the von Neumann algebra $\End(\Theta)$. It is clear that $\Psi$ is a well defined normal $*$-homomorphism. Since the linear span of $S \cH_T S$ is assumed to be dense in $\cH$, we get that $\Psi$ is injective. To check the surjectivity of $\Psi$, it then suffices to show that every projection $p_0 \in \End(\Theta)$ belongs to the image of $\Psi$. Denote by $\cH_0 \subset \cH_T$ the image of $p_0$. Define the $S$-subbimodule $\cH_1 \subset \cH$ as the closed linear span of $S \cH_0 S$. Denote by $p_1 \in \End(\bim{S}{\cH}{S})$ the projection of $\cH$ onto $\cH_1$. It suffices to prove that $\Psi(p_1) = p_0$. This means that we have to show that $p_1(\xi) \in \cH_0$ for every $\xi \in \cH_T$.

For every $\xi \in \cH_T$, we have that $p_1(\xi)$ is a $T$-central vector in $\cH_1$. So, $p_1(\xi) \in p_T(\cH_1)$. Since $\Theta(\al) \cH_0 \subset \cH_0$ for all $\al \in \Irr(\cC)$, it follows from Lemma \ref{lem.projection-central-vectors} that $p_T(\cH_1) \subset \cH_0$. So $p_1(\xi) \in \cH_0$ and the theorem is proved.
\end{proof}

\begin{remark}\label{rem.all-prop-subf}
A first attempt at defining a representation theory for II$_1$ subfactors of finite index and their standard invariants ($\lambda$-lattices\footnote{Recall that a $\lambda$-lattice is a system of multimatrix algebras $(A_{nm})_{0 \leq n \leq m}$, together with natural inclusions $A_{nm} \subset A_{ab}$ if $a \leq n \leq m \leq b$, containing Jones projections, equipped with tracial states and satisfying the axioms \cite[1.1.1--3 and 2.1.1]{Po94b}. Also note that the terminology used in \cite{Po94b} is ``standard $\lambda$-lattice'', but we use the shorter terminology $\lambda$-lattice throughout this paper.}) was given in \cite{Po94a}--\cite{Po99}.
It combined the idea that the representation theory of a ``group-like'' object $\cG$ that can act outerly on a factor $Q$ is the same as the representation theory of the inclusion $Q\subset Q\rtimes \cG=P$, with the fact that for an extremal subfactor of finite index $N\subset M$ with standard  invariant $\cG$, the natural crossed product construction by $\cG$ was found to be the \SE-inclusion $M\ovt M\op \subset M \boxt_{e_N} M\op$. (Following \cite{Po86}, by a representation of an inclusion $Q\subset P$, we mean a ``pointed'' version of Connes' correspondences, namely a Hilbert $P$-bimodule with a cyclic, $Q$-central vector, while the associated ``positive definite functions'' are the normal $Q$-bimodular cp maps on $P$).

Thus, with the above terminology, the representations of $N\subset M$ are defined to be its \SE-correspondences (with the cp \SE-multipliers being the analogue of positive definite functions), and the representations of an abstract $\lambda$-lattice $\cG$ are the \SE-correspondences of a subfactor $N\subset M$ that has $\cG$ as standard invariant.
By \cite{Po94b}, one can associate to $\cG$ a canonical subfactor $N\subset M$ with $\cG_{N,M}=\cG$ (e.g., by taking the ``initial data'' in the amalgamated free product construction of \cite{Po94b} to always be the free group factor $L(\F_\infty)$), thus making this well defined. Instead in \cite[Section 9]{Po99}, one proves that there is an equivalence between the \SE-correspondences (and cp \SE-multipliers) of any two subfactors $N \subset M$ that have the same standard invariant $\cG$, thus allowing to define the representations of $\cG$ as the equivalence class of these objects.

However, this definition of a representation of an abstract $\lambda$-lattice $\cG$ is not entirely satisfactory, as it is not canonical and is not intrinsic in terms of $\cG$,
requiring permanently to go back to a ``supporting subfactor'' whenever used.

Nevertheless, due to the equivalence between the cp \SE-multipliers of any two subfactors having the same standard invariant, established in \cite[Section 9]{Po99}, the approach in  \cite{Po99} allowed defining several rigidity and approximation properties for $\lambda$-lattices, that we recall below.

Thus, we fix an extremal subfactor $N \subset M$ with standard invariant $\cG_{N,M}=\cG$ and \SE-inclusion $T \subset S$. We say that a net of cb \SE-multipliers $\psi_n : S \recht S$ converges pointwise to the identity if $\lim_n \|\psi_n(x) - x\|_2 = 0$ for all $x \in S$. We say that $\psi_n$ converges uniformly to the identity if $\lim_n (\sup_{\|x\| \leq 1} \|\psi_n(x) - x \|_2) = 0$. Finally, we say that $\psi_n$ has finite rank if $\psi_n(S)$ is a finitely generated $T$-bimodule.

\begin{enumlist}
\item The $\lambda$-lattice $\cG$ is \emph{amenable} if and only if there exists a net of cp \SE-multipliers $\psi_n$ that converges to the identity pointwise and such that every $\psi_n$ has finite rank. This is not the usual definition of amenability (see \cite{Po92,Po94a} and \cite[Theorem 5.3]{Po99}), but we will explain in Propositions \ref{prop.all-equiv} and \ref{prop.equiv-amen} why this is an equivalent definition.

\item In \cite[Section 9]{Po99}, the $\lambda$-lattice $\cG$ is said to have \emph{property~\pT} if the following holds: whenever a net of cp \SE-multipliers $\psi_n$ converges to the identity pointwise, it must converge to the identity uniformly.

\item In \cite[Remark 3.5.5]{Po01}, the $\lambda$-lattice $\cG$ is said to have the \emph{Haagerup property} if there exists a net of cp \SE-multipliers $\psi_n$ that converges to the identity pointwise and such that every $\psi_n$, viewed as an element of $T' \cap \langle S,e_T \rangle$, belongs to the compact ideal space given as the norm closed linear span of $S e_T S$.

\item It has been shown in \cite{Br14} that the proof in \cite{Po99} of the equivalence between cp \SE-multipliers of any two subfactors with the same
invariant $\cG$, works equally well to derive the equivalence between their cb \SE-multipliers. Using this, in \cite{Br14}, the $\lambda$-lattice $\cG$ is called \emph{weakly amenable} if there exists a net of cb \SE-multipliers $\psi_n$ that converges to the identity pointwise, such that every $\psi_n$ has finite rank and such that $\limsup_n \|\psi_n\|\cb < \infty$. The smallest possible value of $\limsup_n \|\psi_n\|\cb$ is called the \emph{Cowling-Haagerup constant} $\Lambda(\cG_{N,M})$. If $\Lambda(\cG_{N,M}) = 1$, we say that the $\lambda$-lattice has the \emph{complete metric approximation property}.
\end{enumlist}

In the next section, we will see how to define these rigidity and approximation properties for a $\lambda$-lattice $\cG$ intrinsically in terms of $\cG$, by a direct analogy with
group theory, via appropriate notions of representations and positive definite functions of $\cG$. We will in fact also define these properties for arbitrary rigid C$^*$-tensor categories.
\end{remark}

\section{\boldmath Multipliers on $\lambda$-lattices and C$^*$-tensor categories}\label{sec.mult-tensor-cat}

Let $N \subset M$ be a finite index subfactor and consider the Jones tower $N \subset M \subset M_1 \subset M_2 \subset \cdots$. We write $M_0 = M$ and $M_{-1} = N$. Choosing a tunnel
$$\cdots \subset M_{-2} \subset M_{-1} \subset M_0 \subset M_1 \subset \cdots \; ,$$
we define the \emph{extended standard invariant} $\cGtil_{N,M}$ as $(M_n' \cap M_m)_{n \leq m}$, which can be viewed as the quantum double of $\cG_{N,M}$, in the same way as we viewed the \SE-inclusion $M \ovt M\op \subset M \boxtimes_{e_N} M\op$ as the quantum double of $N \subset M$. By \cite[Corollary 1.8]{PP84}, the extended standard invariant does not depend on the choice of the tunnel.

Abstractly, an \emph{extended $\lambda$-lattice} is defined in exactly the same way as a $\lambda$-lattice \cite{Po94b} (note that in \cite{Po94b}, the terminology used is standard $\lambda$-lattice, but we use the shorter terminology $\lambda$-lattice and extended $\lambda$-lattice throughout our paper)~: it is a system of multimatrix algebras $(A_{nm})_{n \leq m}$, together with natural inclusions $A_{nm} \subset A_{ab}$ if $a \leq n \leq m \leq b$, containing Jones projections, equipped with tracial states and satisfying the axioms \cite[1.1.1--3 and 2.1.1]{Po94b} with the only difference being that the indices range over $\Z$ rather than $\N$.

We now define multipliers (of positive type, completely bounded) of an abstract extended \mbox{$\lambda$-lattice} $\cGtil$. In Proposition \ref{prop.ext-mult}, we will see that in the case where $\cGtil$ is the extended standard invariant of an extremal subfactor $N \subset M$, then there is a natural bijective correspondence between multipliers of the extended $\lambda$-lattice $\cGtil$ and \SE-multipliers of $N \subset M$ in the sense of Definition \ref{def.mult-subf}.

\begin{definition}\label{def.mult-si}
A \emph{multiplier} of an extended $\lambda$-lattice $\cGtil = (A_{nm})_{n \leq m}$ is a family of linear maps
$$\theta_{n,m} : A_{nm} \recht A_{nm} \quad\text{for all}\quad n \leq 0 \leq m \; ,$$
that are $A_{n0} \vee A_{0m}$-bimodular and that are compatible with the inclusions $A_{nm} \subset A_{ab}$ if $a \leq n$ and $b \geq m$.

The multiplier $\theta$ is said to be of \emph{positive type} if all $\theta_{n,m}$ are completely positive maps. We then say that $\theta$ is a \emph{cp-multiplier} of $\cGtil$.

The multiplier $\theta$ is said to be \emph{completely bounded} if $\|\theta\|\cb := \sup_{n,m} \|\theta_{n,m}\|\cb < \infty$. We then say that $\theta$ is a \emph{cb-multiplier} of $\cGtil$.
\end{definition}

Note that every cp-multiplier is completely bounded with $\|\theta\|\cb = \|\theta_{n,m}(1)\|$ for all $n,m$.

\begin{remark}\label{rem.how-to-extend}
Let $\cG = (A_{nm})_{0 \leq n \leq m}$ be an abstract $\lambda$-lattice. We have the natural shift maps $\sigma_a : A_{nm} \recht A_{n+a,m+a}$ that are bijective isomorphisms for $a$ even and bijective anti-isomorphisms for $a$ odd. Therefore, we can canonically associate an extended $\lambda$-lattice $\cGtil = (A_{nm})_{n \leq m}$ to $\cG$, as follows. As algebras, we define $A_{nm} \cong A_{0,-n+m}$ if $n$ is even and $A_{nm} = A_{0,-n+m}\op$ if $n$ is odd, but we use the shift (anti)-isomorphisms to define the inclusions $A_{nm} \subset A_{ab}$ when $a \leq n \leq m \leq b$.

Therefore, Definition \ref{def.mult-si} provides an intrinsic definition of cp-multipliers and cb-multipliers for an abstract $\lambda$-lattice $\cG$.

Denote by $A_{-\infty,+\infty}$ the tracial von Neumann algebra given as the direct limit of all $A_{nm}$ when $n \recht -\infty$ and $m \recht +\infty$. We similarly define $A_{-\infty,0}$ and $A_{0,+\infty}$. We write $T_0 := A_{-\infty,0} \vee A_{0,+\infty}$ and $S_0 := A_{-\infty,+\infty}$. A multiplier $\theta$ is called finitely supported if there exist $n_0 \leq m_0$ such that for all $n \leq m$, we have $\theta(A_{nm}) \subset T_0 A_{n_0,m_0} T_0$. A cp-multiplier $\theta$ is called $c_0$ if, viewed as an element of $\langle S_0,e_{T_0}\rangle$, $\theta$ belongs to the norm closure of $S_0 e_{T_0} S_0$.

We can now copy the end of Remark \ref{rem.all-prop-subf} and define, mutatis mutandis, (weak) amenability, the Haagerup property, CMAP and property~(T) directly for the $\lambda$-lattice $\cG$. All this will become more transparent in this and following sections, using the language of C$^*$-tensor categories.
\end{remark}

Let $N \subset M$ be an extremal subfactor. We denote by $M_\infty$ the direct limit II$_1$ factor $M \subset M_1 \subset M_2 \subset \cdots \subset M_\infty$. We view all relative commutants $M_n' \cap M_m$ as subalgebras of $M_\infty$. As in Section~\ref{sec.rep-theory}, we denote by $M \boxt_{e_N} M\op$ the symmetric enveloping algebra. We have a natural inclusion $M_\infty \subset M \boxt_{e_N} M\op$ so that again, all relative commutants $M_n' \cap M_m$ are subalgebras of $M \boxt_{e_N} M\op$ in a natural way.

\begin{proposition}\label{prop.ext-mult}
Let $N \subset M$ be an extremal subfactor of finite index $[M:N] = \lambda^{-1}$. Denote by $A_{nm} = M_n' \cap M_m$ the associated extended $\lambda$-lattice.

By restricting maps on $M \boxt_{e_N} M\op$ to $M_\infty$ and then to $A_{nm}$, we obtain a bijective, $\|\,\cdot\,\|\cb$-preserving correspondence between
\begin{itemlist}
\item cp \SE-multipliers, resp.\ cb \SE-multipliers, on the subfactor $N \subset M$ in the sense of Definition \ref{def.mult-subf},
\item completely positive, resp.\ completely bounded, normal linear maps $\psi : M_\infty \recht M_\infty$ that are $M \vee (M' \cap M_\infty)$-bimodular and that satisfy $\psi(M_n) \subset M_n$ for all $n \geq 0$,
\item cp-multipliers, resp.\ cb-multipliers on the extended $\lambda$-lattice $(A_{nm})_{n \leq m}$.
\end{itemlist}
\end{proposition}

Before proving this proposition, we look closer into the structure of multipliers on an extended $\lambda$-lattice.

The extended standard invariant $(M_n' \cap M_m)_{n \leq m}$ can be interpreted as follows in the language of C$^*$-tensor categories. In the same way as at the end of Section \ref{sec.rep-theory}, we consider the category $\cC$ of all $M$-bimodules that are isomorphic to a finite direct sum of $M$-subbimodules of $\bim{M}{L^2(M_n)}{M}$ for some $n$. For every $n \geq 1$, we identify
$$L^2(M_n) \cong \underbrace{L^2(M_1) \ot_M \cdots \ot_M L^2(M_1)}_{\text{$n$ factors}} \; .$$
We also have that
\begin{equation}\label{eq.link-si-end}
M_{-2n}' \cap M_{2n} \cong \End\bigl( \bim{M}{L^2(M_n) \ot_M L^2(M_n)}{M} \bigr)
\end{equation}
and under this isomorphism $M_{-2n}' \cap M$ corresponds to $\End(\bim{M}{L^2(M_n)}{M}) \ot 1$, while $M' \cap M_{2n}$ corresponds to $1 \ot \End(\bim{M}{L^2(M_n)}{M})$.

The category $\cC$ of $M$-bimodules generated by a finite index subfactor $N \subset M$ is a \emph{rigid C$^*$-tensor category.}
%
%a C$^*$-tensor category that is semisimple, with irreducible tensor unit $\eps \in \cC$ and with every object $\al \in \cC$ having an adjoint $\albar \in \cC$ that is both a left and a right dual of $\al$. For basic definitions and results on rigid C$^*$-tensor categories, we refer to \cite[Sections 2.1 and 2.2]{NT13}.
%
Entirely similarly as in Definition \ref{def.mult-si}, we can then define multipliers, as well as cp- and cb-multipliers, on a general rigid C$^*$-tensor category.

\begin{definition}\label{def.mult-tc}
A \emph{multiplier} on a rigid C$^*$-tensor category $\cC$ is a family of linear maps
$$\theta_{\al,\be} : \End(\al \ot \be) \recht \End(\al \ot \be) \quad \text{for all} \;\; \al,\be \in \cC$$
that satisfy the following compatibility conditions.
\begin{equation}\label{eq.compat-multipl}
\begin{split}
\theta_{\al_2,\be_2}(U X V^*) &= U \, \theta_{\al_1,\be_1}(X) \, V^* \;\;,\\
\theta_{\al_2 \ot \al_1,\be_1 \ot \be_2}(1 \ot X \ot 1) &= 1 \ot \theta_{\al_1,\be_1}(X) \ot 1 \;\; ,
\end{split}
\end{equation}
for all $\al_i,\be_i \in \cC$ and all $X \in \End(\al_1 \ot \be_1)$ and $U,V \in \Mor(\al_2,\al_1) \ot \Mor(\be_2,\be_1)$. In particular, all $\theta_{\al,\be}$ are $(\End(\al) \ot \End(\be))$-bimodular

%and that are as follows compatible with inclusions:
%$$\theta_{\al_1 \ot \al_2,\be_2 \ot \be_1}(1 \ot X \ot 1) = 1 \ot \theta_{\al_2,\be_2}(X) \ot 1 \quad\text{for all $\al_i,\be_i \in \cC$, $X \in \End(\al_2 \ot \be_2)$.}$$

The multiplier $\theta$ is said to be of \emph{positive type} if all $\theta_{\al,\be}$ are completely positive maps. We then call $\theta$ a \emph{cp-multiplier}.

The multiplier $\theta$ is said to be \emph{completely bounded} if $\|\theta\|\cb := \sup_{\al,\be \in \cC} \|\theta_{\al,\be}\|\cb < \infty$. We then call $\theta$ a \emph{cb-multiplier}.
\end{definition}

\begin{proposition}\label{prop.from-si-to-tc}
Let $N \subset M$ be an extremal subfactor of index $[M:N] = \lambda^{-1}$. Denote by $A_{nm} = M_n' \cap M_m$ the associated extended $\lambda$-lattice, and by $\cC$ the category of $M$-bimodules generated by $N \subset M$ as above.

There is a unique bijective correspondence between multipliers on $(A_{nm})$ (in the sense of Definition \ref{def.mult-si}) and multipliers on $\cC$ (in the sense of Definition \ref{def.mult-tc}) that is compatible with the isomorphism \eqref{eq.link-si-end}. This bijective correspondence preserves being of positive type, being completely bounded, and the $\|\,\cdot\,\|\cb$-norm.
\end{proposition}

\begin{proof}
First assume that $\theta_{nm}$ is a multiplier on the $\lambda$-lattice $(A_{nm})$. Given $\al,\be \in \cC$, we can choose $n$ large enough and projections $p,q \in \End(\bim{M}{L^2(M_n)}{M})$ such that $\al \cong p L^2(M_n)$ and $\be \cong q L^2(M_n)$ as $M$-bimodules. Using \eqref{eq.link-si-end}, we can then view $p \in A_{-2n,0}$, $q \in A_{0,2n}$ and $\End(\al \ot \be) \cong pq A_{-2n,2n} pq$. We define $\theta_{\al,\be}$ by restricting $\theta_{-2n,2n}$. It is easy to check that $\theta_{\al,\be}$ is an unambiguously defined multiplier.

Conversely, assume that $\theta_{\al,\be}$ is a multiplier on $\cC$. By \eqref{eq.link-si-end}, we get the system of maps $\theta_{-2n,2n}$ on $A_{-2n,2n}$ that are $A_{-2n,0} \vee A_{0,2n}$-bimodular. We must show that $\theta_{-2n,2n}(A_{ab}) \subset A_{ab}$ for all $-2n \leq a \leq 0 \leq b \leq 2n$. Once this is proved, we obtain an unambiguously defined multiplier on the $\lambda$-lattice $(A_{nm})$. Take $x \in A_{ab}$. We have to prove that $y:=\theta_{-2n,2n}(x)$ belongs to $A_{ab}$. The inclusion $M_b \subset M_{2n} \subset M_{4n-b}$ is a basic construction. Denote by $e \in M_{4n-b}$ the corresponding Jones projection. Note that $e \in A_{0,4n}$ and that $xe = ex$. We therefore get that
$$y e = \theta_{-4n,4n}(x) \, e = \theta_{-4n,4n}(x e) = \theta_{-4n,4n}(e x) = e y \; .$$
It follows that $y \in M_b$. Similarly exploiting the basic construction $M_{-4n-a} \subset M_{-2n} \subset M_a$, we get that $y \in A_{ab}$.

It is straightforward to see that the above correspondence preserves being of positive type, being completely bounded, and the $\|\,\cdot\,\|\cb$-norm.
\end{proof}

As we explain now, multipliers on a rigid C$^*$-tensor category are exactly labeled by functions $\Irr(\cC) \recht \C$, where $\Irr(\cC)$ is the set of equivalence classes of irreducible objects in $\cC$ (see also Lemma \ref{lem.all-the-same} for a reinterpretation in the context of \SE-correspondences).

In the proof of the following proposition, and also in later sections, we use the following notations. For every $\al \in \cC$, we choose a standard solution of the conjugate equations (in the sense of \cite{LR95}, see also \cite[Definition 2.2.12]{NT13}): $s_\al \in \Mor(\al \ot \albar,\eps)$ and $t_\al \in \Mor(\albar \ot \al,\eps)$ such that
\begin{equation}\label{eq.standard-sol}
(t_\al^* \ot 1) (1 \ot s_\al) = 1 \;\; , \;\; (s_\al^* \ot 1)(1 \ot t_\al)=1 \quad\text{and}\quad t_\al^*(1 \ot X) t_\al = s_\al^* (X \ot 1) s_\al
\end{equation}
for all $X \in \End(\al)$. These $s_\al,t_\al$ are unique up to unitary equivalence and the functional $\Tr_\al(X) = t_\al^*(1 \ot X) t_\al = s_\al^* (X \ot 1) s_\al$ on $\End(\al)$ is uniquely determined and tracial. The trace $\Tr_\al$ is non-normalized: $\Tr_\al(1) = d(\al)$, the categorical dimension of $\al$.

\begin{proposition}\label{prop.descr-multi}
A multiplier $\theta$ on a rigid C$^*$-tensor category $\cC$ can be uniquely extended to a family of linear maps
$$\theta_{\al \ot \be,\gamma \ot \delta} : \Mor(\al \ot \be,\gamma \ot \delta) \recht \Mor(\al \ot \be,\gamma \ot \delta) \quad\text{for all}\;\; \al,\be,\gamma,\delta \in \cC$$
satisfying
$$\theta_{\al_2 \ot \be_2,\gamma_2 \ot \delta_2}(U X V) =  U \; \theta_{\al_1 \ot \be_1,\gamma_1 \ot \delta_1}(X) \; V$$ for all $X \in \Mor(\al_1 \ot \be_1,\gamma_1 \ot \delta_1)$, $U \in \Mor(\al_2,\al_1) \ot \Mor(\be_2,\be_1)$ and $V \in \Mor(\gamma_1,\gamma_2) \ot \Mor(\delta_1,\delta_2)$. If $\theta$ is completely bounded, we have $\|\theta_{\al \ot \be,\gamma \ot \delta}\|\cb \leq \|\theta\|\cb$ for all $\al,\be,\gamma,\delta$.

For every $\al \in \Irr(\cC)$, the space $\Mor(\al \ot \albar,\eps)$ is one-dimensional and therefore $\theta_{\al \ot \albar,\eps \ot \eps}$ is given by multiplication with $\vphi(\al) \in \C$.

Conversely, to every map $\vphi : \Irr(\cC) \recht \C$ corresponds a unique multiplier $\theta$ on $\cC$ such that $\theta_{\al \ot \albar,\eps \ot \eps}$ is given by multiplication with $\vphi(\al)$ for every $\al \in \Irr(\cC)$.
\end{proposition}

\begin{proof}
Define $\mu = \al \oplus \gamma$ and $\eta = \be \oplus \delta$. Denote by $u_\al \in \Mor(\mu,\al)$ and $u_\gamma \in \Mor(\mu,\gamma)$ the natural isometries. Similarly define $u_\be$ and $u_\delta$. We can then define
$$\theta_{\al \ot \be,\gamma \ot \delta}(X) = (u_\al^* \ot u_\beta^*) \, \theta_{\mu,\eta}\bigl( (u_\al \ot u_\beta) X (u_\gamma^* \ot u_\delta^*) \bigr) \, (u_\gamma \ot u_\delta) \; .$$
If $\theta$ is completely bounded, we indeed have that $\|\theta_{\al \ot \be,\gamma \ot \delta}\|\cb \leq \|\theta\|\cb$.

Let $\vphi : \Irr(\cC) \recht \C$ be a map. It remains to prove that there is a unique multiplier $\theta$ on $\cC$ such that
$$\theta_{\al \ot \albar,\eps \ot \eps}(v) = \vphi(\al) v \quad\text{for all}\;\;  \al \in \Irr(\cC) , v \in \Mor(\al \ot \albar,\eps) \; .$$
Fix $\al,\be \in \cC$. Using the notation of \eqref{eq.standard-sol}, we consider the linear bijection
\begin{equation}\label{eq.id-end-mor}
\Gamma : \End(\al \ot \be) \recht \Mor(\albar \ot \al \ot \be \ot \bebar,\eps) : \Gamma(X) = (1 \ot X \ot 1)(t_\al \ot s_\be) \; .
\end{equation}
For every $\delta \in \cC$ and $\pi \in \Irr(\cC)$, denote by $P^\delta_\pi \in \End(\delta)$ the projection onto the direct sum of all subobjects of $\delta$ that are isomorphic with $\pi$. Since $\Mor(\pi_1 \ot \overline{\pi_2},\eps) = \{0\}$ if $\pi_1,\pi_2$ are distinct elements of $\Irr(\cC)$, we get that
$$(P^\delta_\pi \ot 1) v = (1 \ot P^\eta_{\pibar})v \quad\text{for all}\;\; \delta,\eta \in \cC, \pi \in \Irr(\cC) , v \in \Mor(\delta \ot \eta,\eps) \; .$$
We can then define the linear map $\theta_{\al,\be} : \End(\al \ot \be) \recht \End(\al \ot \be)$ by the formula
\begin{align}
\theta_{\al,\be}(X) &= \sum_{\pi \in \Irr(\cC)} \vphi(\pi) \; \Gamma^{-1} \bigl( (P^{\albar \ot \al}_\pi \ot 1 \ot 1) \Gamma(X) \bigr) \label{eq.first-def}\\
&= \sum_{\pi \in \Irr(\cC)} \vphi(\pi) \; \Gamma^{-1} \bigl( (1 \ot 1 \ot P^{\be \ot \bebar}_{\pibar}) \Gamma(X) \bigr) \; . \label{eq.second-def}
\end{align}
Using the bijective linear map
$$\Phi : \End(\al \ot \be) \recht \Mor(\albar \ot \al \ot \be,\be) : \Phi(X) = (1 \ot X)(t_\al \ot 1) \; ,$$
it follows from \eqref{eq.first-def} that
$$\theta_{\al,\be}(X) = \sum_{\pi \in \Irr(\cC)} \vphi(\pi) \; \Phi^{-1}\bigl( (P^{\albar \ot \al}_\pi \ot 1) \Phi(X) \bigr) \; .$$
This formula implies that $\theta_{\al,\be}$ does not depend on the choice of $s_\be$ and that the ``right'' half of the compatibility relations in \eqref{eq.compat-multipl} hold:
\begin{align*}
\theta_{\al,\be_2}\bigl((1 \ot U) X (1 \ot V^*)\bigr) &= (1 \ot U) \, \theta_{\al,\be_1}(X) \, (1 \ot V^*) \;\; ,\\
\theta_{\al,\be_1 \ot \be_2}(X \ot 1) &= \theta_{\al,\be_1}(X) \ot 1 \;\;,
\end{align*}
for all $X \in \End(\al \ot \be_1)$ and $U,V \in \Mor(\be_2,\be_1)$.
Similarly using \eqref{eq.second-def}, we get that the maps $\theta_{\al,\be}$ do not depend on the choice of $t_\al$ and satisfy the ``left'' half of the compatibility relations in \eqref{eq.compat-multipl}. So we have found a multiplier $\theta$ on $\cC$. Its uniqueness is obvious from the above construction.
\end{proof}

We finally prove Proposition \ref{prop.ext-mult}.

\begin{proof}[Proof of Proposition \ref{prop.ext-mult}]
Denote by $T \subset S$ the \SE-inclusion associated with $N \subset M$. By \cite[Proposition 2.6]{Po99}, we have $M_n = (M_{-n}\op)' \cap S$ for all $n \in \Z$. Therefore, every normal $T$-bimodular map $\psi : S \recht S$ satisfies $\psi(M_n) \subset M_n$ for all $n \in \N$ and thus, $\psi(M_\infty) \subset M_\infty$.

Fix a cb-multiplier $\theta$ on $(A_{nm})_{n \leq m}$. The only remaining nontrivial statement to prove is that $\theta$ uniquely extends to a normal completely bounded $T$-bimodular map $\psi : S \recht S$ with $\|\psi\|\cb = \|\theta\|\cb$. Denote by $\cC$ the category of $M$-bimodules generated by $N \subset M$. Consider the function $\vphi : \Irr(\cC) \recht \C$ determining $\theta$ and given by Propositions \ref{prop.from-si-to-tc} and \ref{prop.descr-multi}. Since $\theta_{\pi \ot \pibar,\eps \ot \eps}$ acts by multiplication with $\vphi(\pi)$ on $\Mor(\pi \ot \pibar,\eps)$, it follows that $|\vphi(\pi)| \leq \|\theta\|\cb$ for all $\pi \in \Irr(\cC)$.

The finite dimensional C$^*$-algebras $\End(\eta)$, $\eta \in \cC$, are equipped with the canonical categorical tracial state $d(\eta)^{-1} \Tr_\eta$ and corresponding $\|\,\cdot\,\|_2$-norm. Since the subfactor $N \subset M$ is extremal, the isomorphism \eqref{eq.link-si-end} is trace preserving. Using standard solutions $t_\al \in \Mor(\albar \ot \al,\eps)$ and $s_\be \in \Mor(\be \ot \bebar,\eps)$ of the conjugate equations as in \eqref{eq.standard-sol}, the bijection $\Gamma$ in \eqref{eq.id-end-mor} satisfies $\|\Gamma(X)\| = \sqrt{d(\al)\,d(\be)} \, \|X\|_2$ for all $X \in \End(\al \ot \be)$. Using \eqref{eq.first-def}, we then conclude that
\begin{equation}\label{eq.bound-L2}
\|\theta_{nm}(x)\|_2 \leq \|\theta\|\cb \, \|x\|_2 \quad\text{for all}\;\; n \leq m, x \in A_{nm} \; .
\end{equation}
Define $S_0 := A_{-\infty,+\infty}$ as the tracial von Neumann algebra obtained as the direct limit of the $A_{nm}$ with $n \recht -\infty$ and $m \recht +\infty$. We also consider $T_0 \subset S_0$ defined as $T_0 := A_{-\infty,0} \vee A_{0,+\infty}$, where $A_{-\infty,0}$ is generated by all $A_{n0}$, $n \leq 0$ and $A_{0,+\infty}$ by all $A_{0n}$, $n \geq 0$.

Since $\theta_{nm} : A_{nm} \recht A_{nm}$ is a compatible family of linear maps, with $\|\theta_{nm}\|\cb \leq \|\theta\|\cb < \infty$ for all $n,m$ and satisfying \eqref{eq.bound-L2}, we obtain a unique normal completely bounded map $\psi_0 : S_0 \recht S_0$ such that $\|\psi_0\|\cb = \|\theta\|\cb$, such that the restriction of $\psi_0$ to $A_{nm}$ equals $\theta_{nm}$ and such that $\psi_0$ is $T_0$-bimodular. To extend $\psi_0$ to a $T$-bimodular linear map $\psi : S \recht S$, we apply \cite[Lemma 9.2]{Po99} to the commuting square
\begin{equation}\label{eq.ourcs}
\coms{T}{S}{T_0}{S_0}
\end{equation}
We first note that \cite[Lemma 9.2]{Po99} works equally well for completely bounded maps as for completely positive maps and provides completely bounded extensions without increasing the $\|\,\cdot\,\|\cb$-norm. To apply \cite[Lemma 9.2]{Po99}, we first have to check that the commuting square \eqref{eq.ourcs} is nondegenerate. For every $n \in \Z$, we denote by $e_n \in M_{n+1}$ the Jones projection for the basic construction $M_{n-1} \subset M_{n} \subset M_{n+1}$. Define $P_n$ as the von Neumann algebra generated by $\{e_k \mid k \leq n-1\}$. By \cite[Theorem 4.1.1]{Jo82}, for all $n \leq m$, we have that $P_n \subset P_m$ is a subfactor with the same index as $M_n \subset M_m$. Therefore, the commuting square
\begin{equation}\label{eq.basiccs}
\coms{M_n}{M_m}{P_n}{P_m}
\end{equation}
is nondegenerate. Note that $P_m \subset S_0$ for every $m$. Taking $n=0$ and letting $m \recht +\infty$ in \eqref{eq.basiccs}, we conclude that
$$\coms{M}{M_\infty}{A_{-\infty,0}}{S_0} \qquad\text{and thus also}\qquad \coms{T}{S}{T_0}{S_0}$$
are nondegenerate commuting squares.

By definition, the union $\bigcup_{n \geq 0} A_{-n,n}$ is dense in $S_0$. We claim that for every fixed $n \geq 0$, there exists a basis of $T$ over $T_0$ that commutes with $A_{-n,n}$. Once this claim is proved, the proposition follows from \cite[Lemma 9.2]{Po99} because $\psi_0(A_{-n,n}) \subset A_{-n,n}$. To prove the claim, we use that \eqref{eq.basiccs} is a nondegenerate commuting square for $-n$ and $0$. Therefore, also
$$\coms{M_{-n} \vee M_{-n}\op}{T}{A_{-\infty,-n} \vee A_{n,\infty}}{T_0}$$
is a nondegenerate commuting square for every $n \geq 0$. It then follows that there exists a basis of $T$ over $T_0$ that belongs to $M_{-n} \vee M_{-n}\op$. Inside $T$, we have that $M_{-n} \vee M_{-n}\op$ commutes with $A_{-n,n}$ and so, the claim is proved.
\end{proof}

The following lemma will be useful later: to check that a multiplier $\theta$ on a rigid C$^*$-tensor category $\cC$ is of positive type, it suffices to check the positivity on a specific set of operators. In the formulation of the lemma, we use the standard solutions of the conjugate equations as in \eqref{eq.standard-sol}.

\begin{lemma}\label{lem.pos-check}
Let $\cC$ be a rigid C$^*$-tensor category and $\theta$ a multiplier on $\cC$. The following conditions are equivalent.
\begin{enumlist}
\item For all $\al,\be \in \cC$, the map $\theta_{\al,\be} : \End(\al \ot \be) \recht \End(\al \ot \be)$ is completely positive.
\item For all $\al,\be \in \cC$, the map $\theta_{\al,\be} : \End(\al \ot \be) \recht \End(\al \ot \be)$ is positive.
\item For all $\al \in \cC$, we have that $\theta_{\al,\albar}(s_\al s_\al^*)$ is a positive element in $\End(\al \ot \albar)$.
\end{enumlist}
\end{lemma}
\begin{proof}
1 $\Rightarrow$ 2 $\Rightarrow$ 3 is trivial.

3 $\Rightarrow$ 2. Fix $\al,\be \in \cC$ and $X \in \End(\al \ot \be)$. We must prove that $\theta_{\al,\be}(X^* X) \geq 0$. Define $\pi \in \cC$ as the direct sum of all irreducible subobjects of $\albar \ot \al$ and $\be \ot \bebar$. Then $X = (1 \ot s_\pi^* \ot 1)Y$ for some $Y \in \Mor(\al \ot \pi,\al) \ot \Mor(\pibar \ot \be,\be)$. But then,
$$\theta_{\al,\be}(X^* X) = Y^* (1 \ot \theta_{\pi,\pibar}(s_\pi s_\pi^*) \ot 1) Y \geq 0 \; .$$

2 $\Rightarrow$ 1. Fix $\al,\be \in \cC$ and $n \in \N$. Define $\pi$ as the direct sum of $n$ copies of $\al$. Then, $\End(\pi \ot \be) \cong M_n(\C) \ot \End(\al \ot \be)$ and under this identification, $\theta_{\pi,\be}$ corresponds to $\id \ot \theta_{\al,\be}$.
\end{proof}

Let $\cC$ be a rigid C$^*$-tensor category. A \emph{full C$^*$-tensor subcategory} of $\cC$ can be defined as a subset $\Irr(\cC_1) \subset \Irr(\cC)$ with the property that for all $\al,\be \in \Irr(\cC_1)$, we have that $\albar \in \Irr(\cC_1)$ and that all irreducible subobjects of $\al \ot \be$ belong to $\Irr(\cC_1)$.

The following result can then be interpreted as providing induction of representations of $\cC_1$ to representations of $\cC$.

\begin{proposition}\label{prop.subcat}
Let $\cC$ be a rigid C$^*$-tensor category with full C$^*$-tensor subcategory $\cC_1$.
\begin{enumlist}
\item If $\vphi : \cC \recht \C$ is a cb-multiplier, resp.\ cp-multiplier on $\cC$, then its restriction $\vphi_1$ to $\Irr(\cC_1)$ is a cb-multiplier, resp.\ cp-multiplier on $\cC_1$ with $\|\vphi_1\|\cb \leq \|\vphi\|\cb$.
\item If $\vphi_1 : \cC_1 \recht \C$ is a cp-multiplier on $\cC_1$ and $\vphi : \Irr(\cC) \recht \C$ is defined by $\vphi(\pi) = \vphi_1(\pi)$ if $\pi \in \Irr(\cC_1)$ and $\vphi(\pi) = 0$ if $\pi \not\in \Irr(\cC_1)$, then $\vphi$ is a cp-multiplier on $\cC$.
\end{enumlist}
\end{proposition}
\begin{proof}
The first statement is trivial, because $\theta^{\vphi_1}_{\al,\be}$ is equal to $\theta^\vphi_{\al,\be}$ on $\End(\al \ot \be)$ when $\al,\be \in \cC_1$.

To prove the second statement, we use the linear bijection $\Gamma : \End(\al \ot \be) \recht \Mor(\albar \ot \al \ot \be \ot \bebar,\eps)$ given by \eqref{eq.id-end-mor}. Denote by $P_1^{\albar \ot \al}$ the orthogonal projection in $\End(\albar \ot \al)$ onto the direct sum of all irreducible subobjects of $\albar \ot \al$ that belong to $\Irr(\cC_1)$. Transporting the product and the $*$-operation of $\End(\al \ot \be)$ to $\Mor(\albar \ot \al \ot \be \ot \bebar,\eps)$ via $\Gamma$ and using that $\cC_1$ is a full C$^*$-tensor subcategory of $\cC$, we get that
$$\End_1(\al \ot \be) := \Gamma^{-1}((P_1^{\albar \ot \al} \ot 1 \ot 1)\Mor(\albar \ot \al \ot \be \ot \bebar,\eps))$$
is a unital $*$-subalgebra of $\End(\al \ot \be)$. We denote by $E_1 : \End(\al \ot \be) \recht \End_1(\al \ot \be)$ the unique trace preserving conditional expectation. In particular, $E_1$ is completely positive. Since $\vphi(\pi) = 0$ for all $\pi \not\in \Irr(\cC_1)$, we also have that
$$\theta^\vphi_{\al,\be}(X) = \theta^\vphi_{\al,\be}(E_1(X)) \quad\text{for all}\;\; X \in \End(\al \ot \be) \; .$$
Fixing $\al,\be \in \cC$ and $X \in \End_1(\al \ot \be)$, and using Lemma \ref{lem.pos-check}, it is then sufficient to prove that $\theta^\vphi_{\al,\be}(X^* X) \geq 0$. Define $\pi \in \cC_1$ as the direct sum of all $\pi \in \Irr(\cC_1)$ that appear as a subobject of $\albar \ot \al$ and $\be \ot \bebar$. Since $X \in \End_1(\al \ot \be)$, we can write $X = (1 \ot s_\pi^* \ot 1)Y$ for some $Y \in \Mor(\al \ot \pi,\al) \ot \Mor(\pibar \ot \be,\be)$. But then,
$$\theta^\vphi_{\al,\be}(X^* X) = Y^* (1 \ot \theta^\vphi_{\pi,\pibar}(s_\pi s_\pi^*) \ot 1) Y = Y^* (1 \ot \theta^{\vphi_1}_{\pi,\pibar}(s_\pi s_\pi^*) \ot 1) Y \geq 0 \; .$$
\end{proof}

\section{\boldmath The universal C$^*$-algebra of a rigid C$^*$-tensor category}\label{sec.Cstar-algebra-tensor-cat}

\begin{definition}\label{def.adm-rep-tensor-cat}
Let $\cC$ be a rigid C$^*$-tensor category. A $*$-representation $\Theta : \C[\cC] \recht B(\cK)$ of the fusion $*$-algebra $\C[\cC]$ is called \emph{admissible} if for all $\xi \in \cK$, the map
$$\Irr(\cC) \recht \C : \al \mapsto d(\al)^{-1} \langle \Theta(\al) \xi,\xi \rangle$$
is a cp-multiplier on $\cC$.

By Proposition \ref{prop.pos-mult-to-adm-rep} below, for every admissible $*$-representation, we have $\|\Theta(\al)\| \leq d(\al)$ for all $\al \in \cC$. We can thus define the \emph{universal C$^*$-algebra} $C_u(\cC)$ as the completion of $\C[\cC]$ in a universal admissible $*$-representation.
\end{definition}

Since $\Irr(\cC)$ is a vector space basis of $\C[\cC]$, we consider the bijective correspondence between functions $\vphi : \Irr(\cC) \recht \C$ and linear functionals $\om_\vphi : \C[\cC] \recht \C$ given by $\om_\vphi(\al) = d(\al) \vphi(\al)$ for all $\al \in \Irr(\cC)$.

\begin{proposition}\label{prop.pos-mult-to-adm-rep}
Let $\cC$ be a rigid C$^*$-tensor category and let $\vphi : \Irr(\cC) \recht \C$ be a cp-multiplier on $\cC$. Then, $\om_\vphi$ is a positive functional on $\C[\cC]$ in the sense that $\om_\vphi(x^* x) \geq 0$ for all $x \in \C[\cC]$. Denoting the associated GNS Hilbert space as $\cK_\vphi$, the left multiplication by $x \in \C[\cC]$ extends to a bounded operator $\Theta_\vphi(x)$ on $\cK_\vphi$ with $\|\Theta_\vphi(\al)\| \leq d(\al)$ for all $\al \in \Irr(\cC)$. The $*$-representation $\Theta_\vphi : \C[\cC] \recht B(\cK_\vphi)$ is admissible in the sense of Definition \ref{def.adm-rep-tensor-cat}.
\end{proposition}

Before proving Proposition \ref{prop.pos-mult-to-adm-rep}, we need the following computational lemma.

\begin{lemma}\label{lem.compute-in-tensor-cat}
Let $\cC$ be a rigid C$^*$-tensor category and $\vphi : \Irr(\cC) \recht \C$ a function. For every $x,y \in \C[\cC]$, we define the function $\vphi_{x,y} : \Irr(\cC) \recht \C$ such that
$$\om_{\vphi_{x,y}}(a) = \om_\vphi(y^* a x) \quad\text{for all}\;\; a \in \C[\cC] \; .$$
The multiplier on $\cC$ induced by $\vphi_{x,y}$ in Proposition \ref{prop.descr-multi} is given by
\begin{multline*}
\theta^{\vphi_{x,y}}_{\al,\be}(X) = \\ \sum_{\pi,\eta \in \Irr(\cC)} \, x_\pi \, \overline{y_\eta} \, (1 \ot s_\pi^* \ot 1) \, \theta^\vphi_{(\al \ot \pi) \ot (\pibar \ot \be),(\al \ot \eta) \ot (\etabar \ot \be)} \bigl( (1 \ot s_\pi \ot 1) X (1 \ot s_\eta^* \ot 1)\bigr) \, (1 \ot s_\eta \ot 1)
\end{multline*}
for all $X \in \End(\al \ot \be)$. In particular, if $\vphi$ is a cb-multiplier on $\cC$, then all $\vphi_{x,y}$ are cb-multipliers. And if $\vphi$ is a cp-multiplier on $\cC$, then all $\vphi_{x,x}$ are cp-multipliers.
\end{lemma}

Note that because for all $\rho \in \Irr(\cC)$,
$$y^* \rho x = \sum_{\pi,\eta \in \Irr(\cC)} \, x_\pi \, \overline{y_\eta} \;\; \etabar \, \rho \, \pi
= \sum_{\pi,\eta,\gamma \in \Irr(\cC)} \, x_\pi \, \overline{y_\eta} \, \mult(\gamma,\etabar \ot \rho \ot \pi) \, \gamma \; ,$$
we can alternatively define
$$\vphi_{x,y}(\rho) = \frac{1}{d(\rho)} \sum_{\pi,\eta,\gamma \in \Irr(\cC)} \, x_\pi \, \overline{y_\eta} \, \vphi(\gamma) \, d(\gamma) \, \mult(\gamma,\etabar \ot \rho \ot \pi) \; .$$

\begin{proof}[Proof of Lemma \ref{lem.compute-in-tensor-cat}]
The result follows from a direct computation, using the formulas
\begin{align*}
(t_\al^* \ot 1) (1 \ot P^{\al \ot \be}_\gamma) (t_\al \ot 1) &= \frac{d(\gamma) \, \mult(\gamma,\al \ot \be)}{d(\be)} \, 1 \qquad\text{and} \\
(1 \ot s_\be^*)(P^{\al \ot \be}_\gamma \ot 1)(1 \ot s_\be) &= \frac{d(\gamma) \, \mult(\gamma,\al \ot \be)}{d(\al)} \, 1
\end{align*}
for all $\al,\be,\gamma \in \Irr(\cC)$, where $P^{\al \ot \be}_\gamma \in \End(\al \ot \be)$ denotes the orthogonal projection onto the sum of all subobjects of $\al \ot \be$ that are isomorphic with $\gamma$.
\end{proof}

\begin{proof}[Proof of Proposition \ref{prop.pos-mult-to-adm-rep}]
From Lemma \ref{lem.compute-in-tensor-cat}, it follows that for all $x \in \C[\cC]$,
$$\om_\vphi(x^* x) = \om_{\vphi_{x,x}}(1) = \vphi_{x,x}(\eps) \geq 0 \;  .$$
Denote by $\cK_\vphi$ the GNS Hilbert space given by separation-completion of $\C[\cC]$ with respect to the scalar product $\langle x,y\rangle = \om_\vphi(y^* x)$.

We also know from Lemma \ref{lem.compute-in-tensor-cat} that for all functions $\psi : \Irr(\cC) \recht \C$ and all $\al \in \Irr(\cC)$, we have
$$s_\al^* \, \theta^\psi(s_\al s_\al^*) \, s_\al = \theta^{\psi_{\al,\al}}(1) = \psi_{\al,\al}(\eps) = \om_{\psi_{\al,\al}}(1) = \om_\psi(\albar \, \al) \; .$$
Therefore, for all $x \in \C[\cC]$ and all $\al \in \Irr(\cC)$, we have
$$\om_\vphi(x^* \, \albar \, \al \, x) = \om_{\vphi_{x,x}}(\albar \, \al) = s_\al^* \, \theta^{\vphi_{x,x}}(s_\al s_\al^*) \, s_\al \; .$$
By Lemma \ref{lem.compute-in-tensor-cat}, we get that $\vphi_{x,x}$ is a cp-multiplier for every $x \in \C[\cC]$. So we conclude that
$$\om_\vphi(x^* \, \albar \, \al \, x) \leq \|s_\al\|^4 \, \|\theta^{\vphi_{x,x}}(1)\| = d(\al)^2 \, \vphi_{x,x}(\eps) = d(\al)^2 \, \om_\vphi(x^* x) \; .$$
It follows that left multiplication by $\al \in \Irr(\cC)$ extends to a bounded operator $\Theta_\vphi(\al)$ on $\cK_\vphi$ with $\|\Theta_\vphi(\al)\| \leq d(\al)$.

We already observed above that $\vphi_{x,x}$ is a cp-multiplier for every $x \in \C[\cC]$. Noting that
$$\vphi_{x,x}(\al) = \frac{1}{d(\al)} \, \langle \Theta_{\vphi}(\al) \, x , x \rangle \; ,$$
we find that $\al \mapsto d(\al)^{-1} \, \langle \Theta_\vphi(\al) \xi,\xi \rangle$ is a cp-multiplier for all $\xi$ in the dense subspace of $\cK_\vphi$ given by the image of $\C[\cC]$. When $\vphi_n : \Irr(\cC) \recht \C$ is a sequence of cp-multipliers on $\cC$ and $\vphi_n \recht \vphi$ pointwise, we have that $\theta^{\vphi_n}_{\al,\be} \recht \theta^\vphi_{\al,\be}$ pointwise in norm for all fixed $\al,\be \in \cC$, so that also $\vphi$ is a cp-multiplier. We conclude that $\al \mapsto d(\al)^{-1} \, \langle \Theta_\vphi(\al) \xi,\xi \rangle$ is a cp-multiplier for all $\xi \in \cK_\vphi$ so that $\Theta_\vphi$ is indeed an admissible $*$-representation.
\end{proof}

\begin{corollary}\label{cor.regular-and-triv-rep}
The $*$-representation
$$\Theta_0 : \C[\cC] \recht B(\ell^2(\Irr(\cC))) : \Theta_0(\al) \, \delta_\be = \sum_{\gamma \in \Irr(\cC)} \, \mult(\gamma,\al \ot \be) \, \delta_\gamma$$
is admissible. It is isomorphic with $\Theta_{\vphi_0}$ where $\vphi_0 : \Irr(\cC) \recht \C$ is the cp-multiplier defined by $\vphi_0(\eps) = 1$ and $\vphi_0(\al) = 0$ for all $\al \in \Irr(\cC)$ with $\al \neq \eps$.

We call $\Theta_0$ the \emph{regular representation} of $\C[\cC]$ and we define $C_r(\cC)$ as the closure of $\Theta_0(\C[\cC])$. It follows that $\C[\cC]$ is injectively embedded into the C$^*$-algebra $C_u(\cC)$.

Also the $1$-dimensional $*$-representation given by $\counit : \C[\cC] \recht \C : \counit(\al) = d(\al)$ for all $\al \in \Irr(\cC)$ is admissible. It is isomorphic with $\Theta_{\vphi_\counit}$ where $\vphi_\counit(\al) = 1$ for all $\al \in \Irr(\cC)$. We call $\counit$ the \emph{trivial representation} of $\C[\cC]$.
\end{corollary}
\begin{proof}
A direct computation shows that $\theta^{\vphi_0}_{\al,\be} : \End(\al \ot \be) \recht \End(\al \ot \be)$ is the unique trace preserving conditional expectation of $\End(\al \ot \be)$ onto $\End(\al) \ot \End(\be)$. Therefore, $\vphi_0$ is a cp-multiplier. It can be readily checked that the associated admissible $*$-representation $\Theta_{\vphi_0}$ is precisely $\Theta_0$.

Since $\theta^{\vphi_\counit}_{\al,\be}$ is the identity map, also $\vphi_\counit$ is a cp-multiplier. Its associated admissible $*$-representation is the $1$-dimensional $\counit$.
\end{proof}

We also record the following result. It is an immediate consequence of Proposition \ref{prop.subcat}.

\begin{proposition}\label{prop.Cstar-subcat}
Let $\cC$ be a rigid C$^*$-tensor category with full C$^*$-tensor subcategory $\cC_1$. The natural inclusion $\C[\cC_1] \subset \C[\cC]$ extends to an injective $*$-homomorphism $C_u(\cC_1) \hookrightarrow C_u(\cC)$.
\end{proposition}

The main remaining goal of this section is to prove the following theorem.

\begin{theorem}\label{thm.isom-univ-Cstar}
Let $N \subset M$ be an extremal finite index subfactor with associated \SE-inclusion $T \subset S$ and bimodule category $\cC$. A $*$-representation of $\C[\cC]$ is admissible in the subfactor sense of Definition \ref{def.full-Cstar-subfactor} if and only if it is admissible in the C$^*$-tensor category sense of Definition \ref{def.adm-rep-tensor-cat}. In other words, we have a natural isomorphism $C_u(N \subset M) \cong C_u(\cC)$.
\end{theorem}

Before proving Theorem \ref{thm.isom-univ-Cstar}, we need a lemma making the identifications in Propositions \ref{prop.ext-mult} and \ref{prop.from-si-to-tc} more explicit. As already mentioned above, by \cite[Theorem 4.5]{Po99}, we can uniquely decompose the $T$-bimodule $L^2(S)$ into a direct sum of irreducible $T$-subbimodules $(\cL_\pi)_{\pi \in \Irr(\cC)}$ labeled by the elements $\pi \in \Irr(\cC)$ such that $\cL_\pi \cong \pi \ot \pibar\op$ as $M \ovt M\op$-bimodules. We denote by $\cL^0_\pi \subset \cL_\pi$ the space of $T$-bounded vectors in $\cL_\pi$. Since $T \subset S$ is irreducible, we have that $\cL^0_\pi = \cL_\pi \cap S$.

\begin{lemma}\label{lem.all-the-same}
Let $\vphi : \Irr(\cC) \recht \C$ be a cb-multiplier. Denote by $\psi_\vphi : S \recht S$ the normal completely bounded $T$-bimodular map that is associated with $\vphi$ by combining Propositions \ref{prop.ext-mult} and \ref{prop.from-si-to-tc}. Then $\psi_\vphi(x) = \vphi(\pibar) x$ for all $\pi \in \Irr(\cC)$ and all $x \in \cL^0_\pi$.
\end{lemma}

\begin{proof}
Since the $T$-bimodules $\cL_\pi$, $\pi \in \Irr(\cC)$ are mutually inequivalent and all appear with multiplicity one in $L^2(S)$, and because $\psi_\vphi$ is $T$-bimodular, it follows that $\psi_\vphi$ acts by multiplication with a scalar on each of the $\cL^0_\pi$. We have to prove that this scalar is $\vphi(\pibar)$.

Consider the cb-multiplier $(\theta_{n,m})_{n \leq m}$ on the extended $\lambda$-lattice $(A_{nm})_{n \leq m}$ as given by Proposition \ref{prop.from-si-to-tc}. Fix $\pi \in \Irr(\cC)$ and take $n$ large enough such that $\pi$, and hence also $\pibar$, appear as $M$-subbimodules of $L^2(M_n)$. We now copy a part of the proof of \cite[Theorem 4.5]{Po99}. Denote by $q_1$ the minimal projection in $M_0' \cap M_{2n}$ given by the projection onto $\pi$ as an $M$-submodule of $L^2(M_n)$. Denote by $f_1$ the projection of $L^2(M_n)$ onto $L^2(M_0)$. So, $f_1$ is the Jones projection for the basic construction $M_0 \subset M_n \subset M_{2n}$. Denote by $r_0 \in A_{-n,n}$ the Jones projection for $M_{-n} \subset M_0 \subset M_n$. Finally, denote by $q_{-1}$, resp. $f_{-1}$, the reflections of $q_1$, resp. $f_1$, into $A_{-2n,0}$. Put $v' = q_{-1} q_1 r_0 f_{-1} f_1$. Using the unique trace preserving conditional expectation, we finally define $v = E_{A_{-n,n}}(v')$. In the proof of \cite[Theorem 4.5]{Po99}, it is shown that $\cL_\pi$ equals the closed linear span of $T v T$. So, it is sufficient to prove that $\theta_{-n,n}(v) = \vphi(\pibar) \, v$.

Because $E_{A_{-n,n}} : A_{-2n,2n} \recht A_{-n,n}$ can be implemented by the Jones projections of $M_n \subset M_{2n} \subset M_{3n}$, resp.\ $M_{-3n} \subset M_{-2n} \subset M_{-n}$, we find that $\theta_{-n,n} \circ E_{A_{-n,n}} = E_{A_{-n,n}} \circ \theta_{-2n,2n}$. So, we must prove that $\theta_{-2n,2n}(v') = \vphi(\pibar) \, v'$. But under the isomorphism \eqref{eq.link-si-end}, we get that $v'$ is a nonzero element in $\Mor(\pibar \ot \pi,\eps)$. So by Proposition \ref{prop.descr-multi}, we indeed get that $\theta_{-2n,2n}(v') = \vphi(\pibar) \, v'$.
\end{proof}

\begin{proof}[Proof of Theorem \ref{thm.isom-univ-Cstar}]
Since we can take direct sums of both \SE-correspondences $\bim{S}{\cH}{S}$ and of admissible representations in the sense of Definition \ref{def.adm-rep-tensor-cat}, it suffices to consider the cyclic case. So, we take a cp-multiplier $\vphi : \Irr(\cC) \recht \C$ on $\cC$ with associated cp \SE-multiplier $\psi_\vphi : S \recht S$. By Lemma \ref{lem.all-the-same}, we have $\psi_\vphi(x) = \vphi(\albar) x$ for all $\al \in \Irr(\cC)$ and $x \in \cL^0_\al$. With $\psi_\vphi$, we construct the cyclic \SE-correspondence $\bim{S}{\cH^\vphi}{S}$ with cyclic vector $\xi_0 \in \cH^\vphi_T$ and scalar product $\langle x \xi_0 y , \xi_0 \rangle = \tau(x \psi_\vphi(y))$ for all $x, y \in S$. We consider the associated $*$-representation $\Theta : \C[\cC] \recht B(\cH^\vphi_T)$ given by Theorem \ref{thm.full-Cstar-subfactor}. It follows from Lemma \ref{lem.projection-central-vectors} that $\xi_0$ is a cyclic vector for $\Theta$. We also consider the cyclic representation $\Theta_\vphi : \C[\cC] \recht B(\cK_\vphi)$ given by Proposition \ref{prop.pos-mult-to-adm-rep}. For every $\al \in \Irr(\cC)$, and using an orthonormal basis $v_i \in \cL^0_\al$ of $\cL_\al$ as a right $T$-module, we get that
\begin{align*}
\langle \Theta(\al) \xi_0 , \xi_0 \rangle &= \frac{1}{d(\al)} \sum_i \langle v_i \xi_0 v_i^* , \xi_0 \rangle = \frac{1}{d(\al)} \sum_i \tau(v_i \psi_\vphi(v_i^*))
\\ &= \frac{\vphi(\al)}{d(\al)} \sum_i \tau(v_i v_i^*) = d(\al) \, \vphi(\al) = \om_\vphi(\al) = \langle \Theta_\vphi(\al) 1, 1\rangle \; .
\end{align*}
So the cyclic $*$-representations $\Theta$ and $\Theta_\vphi$ are unitarily conjugate. Since all cp \SE-multipliers $\psi : S \recht S$ are of the form $\psi_\vphi$ for some cp-multiplier $\vphi : \Irr(\cC) \recht \C$, we have proved that both meanings of admissibility coincide on cyclic representations.
\end{proof}

\begin{remark}
Given an admissible representation $\Theta : \C[\cC] \recht B(\cK)$, the cp-multipliers on $\cC$ of the form $\al \mapsto d(\al)^{-1} \langle \Theta(\al)\xi,\xi \rangle$, where $\xi \in \cK$ is a unit vector, are called the \emph{coefficients} of $\Theta$. We say that $\Theta$ is weakly contained in $\Theta'$ if every coefficient of $\Theta$ can be pointwise approximated by convex combinations of coefficients of $\Theta'$, or equivalently, if $\|\Theta(x)\| \leq \|\Theta'(x)\|$ for all $x \in \C[\cC]$.

When $\cC$ is the category of $M$-bimodules generated by an extremal subfactor $N \subset M$ and identifying $C_u(N \subset M) = C_u(\cC)$ through Theorem \ref{thm.isom-univ-Cstar}, the above notion of weak containment is then equivalent with the notion of weak containment for \SE-correspondences introduced in Remark \ref{rem.fell-topology-SE}.
\end{remark}

\section[{\boldmath Approximation and rigidity properties of subfactors and C$^*$-tensor categories}]{\boldmath Approximation and rigidity properties of subfactors and \linebreak C$^*$-tensor categories}\label{sec.approx-rigid-prop}

\begin{definition}\label{def.prop-Cstar-cat}
A rigid C$^*$-tensor category $\cC$ is said
\begin{enumlist}
\item to be \emph{amenable} if there exists a net of finitely supported cp-multipliers $\vphi_n : \Irr(\cC) \recht \C$ that converges to $1$ pointwise~;
\item to have \emph{property~\pT} if every net of cp-multipliers $\vphi_n : \Irr(\cC) \recht \C$ that converges to $1$ pointwise, must converge to $1$ uniformly on $\Irr(\cC)$~;
\item to have the \emph{Haagerup property} if there exists a net of cp-multipliers $\vphi_n : \Irr(\cC) \recht \C$ such that every $\vphi_n$ converges to $0$ at infinity and such that $\vphi_n \recht 1$ pointwise~;
\item to be \emph{weakly amenable} if there exists a net of finitely supported cb-multipliers $\vphi_n : \Irr(\cC) \recht \C$ that converges to $1$ pointwise and such that $\limsup_n \|\vphi_n\|\cb < \infty$. The smallest possible value of $\limsup_n \|\vphi_n\|\cb$ is called the \emph{Cowling-Haagerup constant} $\Lambda(\cC)$ of $\cC$. If $\Lambda(\cC) = 1$, we say that $\cC$ has \emph{CMAP}.
\end{enumlist}
Obviously, if $\Irr(\cC)$ is countable, nets may everywhere be replaced by sequences.
\end{definition}

Combining Propositions \ref{prop.ext-mult}, \ref{prop.from-si-to-tc} and Lemma \ref{lem.all-the-same}, we immediately get the following result.

\begin{proposition}\label{prop.all-equiv}
Let $N \subset M$ be an extremal subfactor with standard invariant $\cG_{N,M}$ and denote by $\cC$ the category of $M$-bimodules generated by $N \subset M$. Then, the $\lambda$-lattice $\cG_{N,M}$ has any of the above rigidity/approximation properties in the sense of Remarks \ref{rem.all-prop-subf} and \ref{rem.how-to-extend} if and only if the category $\cC$ has the corresponding property in the sense of Definition \ref{def.prop-Cstar-cat}.
\end{proposition}

Again, the above definition of amenability is not the usual one (see \cite{Po92,Po94a}), but the following proposition shows that it is equivalent, with a proof following a standard recipe in the theory.

\begin{proposition}\label{prop.equiv-amen}
Let $\cC$ be a rigid C$^*$-tensor category. Consider the reduced C$^*$-algebra $C_r(\cC)$, the regular $*$-representation $\Theta_0$ and the trivial $*$-representation $\counit$ of $\C[\cC]$ introduced in Corollary \ref{cor.regular-and-triv-rep}. Then the following conditions are equivalent.
\begin{enumlist}
\item $\cC$ is amenable in the sense of Definition \ref{def.prop-Cstar-cat}.
\item The natural $*$-homomorphism $C_u(\cC) \recht C_r(\cC)$ is an isomorphism.
\item We have $|\counit(x)| \leq \|\Theta_0(x)\|$ for all $x \in \C[\cC]$.
\item $\cC$ is amenable in the usual sense: $\|\Theta_0(\al)\| = d(\al)$ for all $\al \in \cC$.
\end{enumlist}
\end{proposition}

\begin{proof}
1 $\Rightarrow$ 2. We first prove that if $\vphi : \Irr(\cC) \recht \C$ is a finitely supported cp-multiplier, then there exists a unique positive functional $\Om_\vphi \in C_r(\cC)^*$ satisfying $\Om_\vphi \circ \Theta_0 = \om_\vphi$. Indeed, we define the vector $\xi \in \ell^2(\Irr(\cC))$ given by $\xi(\pi) = d(\pi) \, \overline{\vphi(\pi)}$. We define $\Om_\vphi \in C_r(\cC)^*$ by the formula $\Om_\vphi(T) = \langle T \delta_\eps,\xi \rangle$. By construction, $\Om_\vphi \circ \Theta_0 = \om_\vphi$. It then also follows that $\Om_\vphi$ is a positive functional.

Assume that 1 holds and take a net $\vphi_n : \Irr(\cC) \recht \C$ of finitely supported cp-multipliers on $\cC$ that converges to $1$ pointwise. Let $\vphi : \Irr(\cC) \recht \C$ be an arbitrary cp-multiplier on $\cC$. To prove 2, we have to show that $|\om_\vphi(x)| \leq \vphi(\eps) \, \|\Theta_0(x)\|$ for every $x \in \C[\cC]$. Define $\psi_n = \vphi \vphi_n$ as the pointwise product of $\vphi$ and $\vphi_n$. Since $\theta^{\psi_n}_{\al,\be}$ equals the composition $\theta^\vphi_{\al,\be} \circ \theta^{\vphi_n}_{\al,\be}$, it follows that every $\psi_n$ is a cp-multiplier on $\cC$. By construction, every $\psi_n$ is finitely supported. Using the positive functionals $\Om_{\psi_n}$ defined in the previous paragraph, we find that for every $x \in \C[\cC]$, we have
$$|\om_\vphi(x)| = \lim_n |\om_{\psi_n}(x)| = \lim_n |\Om_{\psi_n}(\Theta_0(x))| \leq \|\Theta_0(x)\| \, \limsup_n \Om_{\psi_n}(1) = \vphi(\eps) \, \|\Theta_0(x)\| \; .$$

2 $\Rightarrow$ 3. Denoting by $\Theta_u : \C[\cC] \recht C_u(\cC)$ the canonical embedding, we get that $|\counit(x)| \leq \|\Theta_u(x)\| = \|\Theta_0(x)\|$ for all $x \in \C[\cC]$.

3 $\Rightarrow$ 4. We get $d(\al) = |\counit(\al)| \leq \|\Theta_0(\al)\| \leq d(\al)$ for all $\al \in \C[\cC]$.

4 $\Rightarrow$ 1. By a standard argument (see e.g.\ \cite[Theorem 4.1]{HI97}), we find a net of unit vectors $\xi_n \in \ell^2(\Irr(\cC))$ such that
$\lim_n \langle \Theta_0(\al) \xi_n , \xi_n \rangle = d(\al)$ for all $\al \in \cC$.
Approximating these vectors $\xi_n$ by finitely supported functions and using the notation of Corollary \ref{cor.regular-and-triv-rep}, we find a sequence $x_n \in \C[\cC]$ such that
$$\lim_n \om_{\vphi_0}(x_n^* \al x_n) = d(\al) \quad\text{for all}\;\; \al \in \cC \; .$$
By Proposition \ref{prop.pos-mult-to-adm-rep}, the functions $\al \mapsto d(\al)^{-1} \om_{\vphi_0}(x_n^* \al x_n)$ are cp-multipliers. By construction, they are finitely supported and converge to $1$ pointwise.
\end{proof}

The following two results on property~(T) are proved in exactly the same way as in the discrete group case.

\begin{proposition}\label{prop.T-finitely-generated}
Let $\cC$ be a rigid C$^*$-tensor category with property~\pT. Then, $\cC$ is finitely generated~: there exists an $\al \in \cC$ such that every $\pi \in \Irr(\cC)$ is isomorphic with a subobject of a tensor power of $\al$.
\end{proposition}
\begin{proof}
Let $\cJ$ be the set of all $\al \in \cC$ with the properties that $\eps$ is contained in $\al$ and that $\albar \cong \al$. For every $\al \in \cJ$, denote by $\cC_\al$ the full C$^*$-tensor subcategory of $\cC$ generated by $\al$. Note that $\pi \in \Irr(\cC_\al)$ if and only if $\pi \in \Irr(\cC)$ and $\pi$ is contained in some tensor power of $\al$. We partially order $\cJ$ by writing $\al \leq \be$ when $\al$ is isomorphic to a subobject of $\be$. For every $\al \in \cJ$, we define $\vphi_\al : \Irr(\cC) \recht \C$ by $\vphi_\al(\pi) = 1$ if $\pi \in \Irr(\cC_\al)$ and $\vphi_\al(\pi)=0$ if $\pi \not\in \Irr(\cC_\al)$. By Proposition \ref{prop.subcat}, every $\vphi_\al$ is a cp-multiplier on $\cC$. By construction, $\vphi_\al \recht 1$ pointwise. Since $\cC$ has property~(T), we get that $\vphi_\al \recht 1$ uniformly. So, there exists an $\al \in \cJ$ with $\cC = \cC_\al$.
\end{proof}

\begin{proposition}\label{prop.char-prop-T}
Let $\cC$ be a rigid C$^*$-tensor category. Then the following conditions are equivalent.
\begin{enumlist}
\item $\cC$ has property~\pT.
\item There exists a nonzero projection $p \in C_u(\cC)$ such that $\al p = d(\al) p$ for all $\al \in \cC$. Such a projection is unique and satisfies $\counit(p) = 1$.
\item If $\om_n \in C_u(\cC)^*$ is a net of states and $\om_n \recht \counit$ weakly$^*$, then $\|\om_n - \counit\| \recht 0$.
\end{enumlist}
\end{proposition}

In the proof of Proposition \ref{prop.char-prop-T}, we use the following lemma.

\begin{lemma}\label{lem.elementary}
Let $\Theta : \C[\cC] \recht B(\cH)$ be an admissible $*$-representation. Assume that $\xi \in \cH$ is a unit vector satisfying
$$\real \langle \Theta(\pi) \xi , \xi \rangle \geq \frac{1}{2} d(\pi) \quad\text{for all}\;\; \pi \in \Irr(\cC) \; .$$
Then there exists a unit vector $\xi_0 \in \cH$ satisfying $\Theta(\al) \xi_0 = d(\al) \xi_0$ for all $\al \in \cC$.
\end{lemma}
\begin{proof}
For every $\al \in \cC$, define the contraction $T_\al \in B(\cH)$ given by $T_\al = d(\al)^{-1} \Theta(\al)$. Define $C \subset \cH$ as the closed convex hull of $\{T_\al \xi \mid \al \in \cC\}$. Denote by $\xi_1 \in C$ the unique element of minimal norm. Since for every $\al \in \cC$, the operator $T_\al$ is a convex combination of operators $T_\pi$, $\pi \in \Irr(\cC)$, we get that $\real \langle \xi_1,\xi \rangle \geq 1/2$. In particular, $\xi_1 \neq 0$. Since $T_\al T_\be = T_{\al \ot \be}$, it follows that $T_\al(C) \subset C$ for all $\al \in \cC$. Since every $T_\al$ is a contraction, the uniqueness of $\xi_1$ implies that $T_\al \xi_1 = \xi_1$ for all $\al \in \cC$. Taking $\xi_0 = \|\xi_1\|^{-1} \, \xi_1$, the lemma is proved.
\end{proof}

\begin{proof}[Proof of Proposition \ref{prop.char-prop-T}]
1 $\Rightarrow$ 2. By Proposition \ref{prop.T-finitely-generated}, we can take $\al_1 \in \cC$ such that $\al_1 \cong \overline{\al_1}$ and such that every $\pi \in \Irr(\cC)$ is contained in some tensor power of $\al_1$. Then $\al_1 \in C_u(\cC)$ is self-adjoint. Since $\|\al_1\| \leq d(\al_1)$, the spectrum $\si(\al_1)$ of $\al_1$ is contained in $[-d(\al_1),d(\al_1)]$. Since $\counit(\al_1) = d(\al_1)$, we have that $d(\al_1) \in \si(\al_1)$. We claim that $d(\al_1)$ is an isolated point of $\si(\al_1)$. Assume that $d(\al_1)$ is not isolated.

To reach a contradiction, choose a faithful unital $*$-representation $\Theta : C_u(\cC) \recht B(\cH)$. For every $n \in \N$, denote by $Q_n \in B(\cH)$ the spectral projection of $\Theta(\al_1)$ corresponding to the open interval $(d(\al_1) - 1/n,d(\al_1))$. Since we assumed that $d(\al_1)$ is not isolated in the spectrum of $\al_1$, we find that $Q_n \neq 0$ for every $n \in \N$. Choose a unit vector $\xi_n \in Q_n \cH$. By construction, $\lim_n \|\Theta(\al_1)\xi_n - d(\al_1) \xi_n\| = 0$. Then also for every $k \in \N$, we have $\lim_n \|\Theta(\al_1^k) \xi_n - d(\al_1^k) \xi_n\| = 0$. By the generating property of $\al_1$, we get that $\lim_n \|\Theta(\al) \xi_n - d(\al)\xi_n\| = 0$ for all $\al \in \cC$.

Define the cp-multipliers $\vphi_n : \Irr(\cC) \recht \C$ given by $\vphi_n(\al) = d(\al)^{-1} \langle \Theta(\al) \xi_n,\xi_n \rangle$ for all $\al \in \cC$. It follows that $\vphi_n \recht 1$ pointwise. Since $\cC$ has property~(T), we conclude that $\vphi_n \recht 1$ uniformly. Fix $n$ large enough such that $\real \vphi_n(\pi) \geq 1/2$ for all $\pi \in \Irr(\cC)$.
Denote by $P \in B(\cH)$ the spectral projection of $\Theta(\al_1)$ corresponding to the singleton $\{d(\al_1)\}$. Then $\Theta(\al_1) P = d(\al_1) P$ and the generating property of $\al_1$ implies that $\Theta(\al) P = d(\al) P$ for all $\al \in \cC$. So, the subspaces $P \cH$ and $(1-P) \cH$ are invariant under $\Theta(C_u(\cC))$. By construction, $\xi_n \in (1-P)\cH$ and
$$\real \langle \Theta(\pi) \xi_n , \xi_n \rangle \geq \frac{1}{2} d(\pi) \quad\text{for all}\;\; \pi \in \Irr(\cC) \; .$$
By Lemma \ref{lem.elementary}, there exists a unit vector $\xi_0 \in (1-P)\cH$ satisfying $\Theta(\al_1) \xi_0 = d(\al_1) \xi_0$ for all $\al \in \cC$. This means that $\xi_0 \in P \cH$ and we reached a contradiction.

So we have proved that $d(\al_1)$ is an isolated point of $\si(\al_1)$. By continuous functional calculus, we then find a nonzero projection $p \in C_u(\cC)$ such that $\al_1 p = d(\al_1) p$. Again using the generating property of $\al_1$, it follows that $\al p = d(\al) p$ for all $\al \in \cC$. This means that $x p = \counit(x) p$ for all $x \in C_u(\cC)$. Taking $x = p$, it follows that $\counit(p) = 1$. If also $x q = \counit(x) q$ for all $x \in C_u(\cC)$, we get that $p q = \counit(p) q = q$ and similarly $q p = p$. So also the uniqueness is proved.

2 $\Rightarrow$ 3. Take the nonzero projection $p \in C_u(\cC)$ such that $x p = \counit(x) p$ for all $x \in C_u(\cC)$. Assume that $\om_n \in C_u(\cC)^*$ is a net of states such that $\om_n \recht \counit$ weakly$^*$. By the Cauchy-Schwarz inequality, we have that
$$\|p \om_n - \om_n\|^2 \leq \om_n(1-p) \recht \counit(1-p) = 0 \; .$$
Since $p \om_n = \om_n(p) \counit$, it follows that $\|\om_n -\counit\| \recht 0$.

3 $\Rightarrow$ 1. Let $\vphi_n : \Irr(\cC) \recht \C$ be a net of cp-multipliers that converges to $1$ pointwise. We have to prove that $\vphi_n \recht 1$ uniformly. We may assume that $\vphi_n(\eps) = 1$ for all $n$. Define the states $\om_n \in C_u(\cC)^*$ given by $\om_n(\al) = d(\al) \vphi_n(\al)$ for all $\al \in \Irr(\cC)$. By construction, $\om_n \recht \counit$ weakly$^*$. So, $\|\om_n - \counit\| \recht 0$. Since $\|\al\| \leq d(\al)$ for all $\al \in \cC$, it follows that $\vphi_n \recht 1$ uniformly.
\end{proof}

Let $\cC$ be a rigid C$^*$-tensor category and $\cC_1$ a full C$^*$-tensor subcategory. When $\cC_1$ has finite index in $\cC$ in a suitable sense, we expect that property~(T) for $\cC$ is equivalent with property~(T) for $\cC_1$. We only need this result in the following particular and easy case.

\begin{proposition}\label{prop.stability-T}
Let $\cC$ be a rigid C$^*$-tensor category, $\Lambda$ a finite group and $\Xi : \Irr(\cC) \recht \Lambda$ a map with the property that $\Xi(\gamma) = \Xi(\al) \, \Xi(\be)$ whenever $\al,\be,\gamma \in \Irr(\cC)$ are such that $\gamma$ is a subobject of $\al \ot \be$. Define the full C$^*$-tensor subcategory $\cC_1$ of $\cC$ given by $\Irr(\cC_1) = \{\al \in \Irr(\cC) \mid \Xi(\al) = e\}$.

Then $\cC$ has property~\pT\ if and only if $\cC_1$ has property~\pT.
\end{proposition}
\begin{proof}
First assume that $\cC$ has property~(T). Let $\vphi_{1,n} : \Irr(\cC_1) \recht \C$ be a net of cp-multipliers on $\cC_1$ that converges to $1$ pointwise. We have to prove that $\vphi_{1,n} \recht 1$ uniformly on $\cC_1$. We may assume that $\vphi_{1,n}(\eps) = 1$ for all $n$. Define $\vphi_n : \Irr(\cC) \recht \C$ by putting $\vphi_n(\pi) = \vphi_{1,n}(\pi)$ if $\pi \in \Irr(\cC_1)$ and $\vphi_n(\pi) = 0$ if $\pi \not\in \Irr(\cC_1)$. By Proposition \ref{prop.subcat}, all $\vphi_n$ are cp-multipliers on $\cC$. Since the image of $\Xi$ is a subgroup of $\Lambda$, we may assume that $\Xi$ is surjective. Choose for every $s \in \Lambda$, an $\al_s \in \Irr(\cC)$ such that $\Xi(\al_s) = s$. For the neutral element $e \in \Lambda$, we choose $\al_e = \eps$. Define the element $a \in \C[\cC]$ given by
$$a = \frac{1}{\sqrt{|\Lambda|}} \sum_{s \in \Lambda} \frac{1}{d(\al_s)} \al_s \; .$$
By Lemma \ref{lem.compute-in-tensor-cat}, we can define the cp-multipliers $\psi_n : \Irr(\cC) \recht \C$ such that $\om_{\psi_n}(x) = \om_{\vphi_n}(a^*xa)$. Since $\vphi_n(\gamma) \recht 1$ pointwise for all $\gamma \in \Irr(\cC_1)$ and $\vphi_n(\gamma) = 0$ for all $\gamma \not\in \Irr(\cC_1)$, we get that $\psi_n \recht \psi$ pointwise, with
$$\psi(\pi) = \frac{1}{|\Lambda| \, d(\pi)} \; \sum_{s,t \in \Lambda, \gamma \in \Irr(\cC_1)} \;\; \frac{1}{d(\al_s) \, d(\al_t)} \, d(\gamma) \, \mult(\gamma, \overline{\al_s} \ot \pi \ot \al_t) \quad\text{for all}\;\; \pi \in \Irr(\cC) \; .$$
Given $s,t \in \Lambda$ and $\pi \in \Irr(\cC)$, we either have $s = \Xi(\pi) t$ so that all subrepresentations of $\overline{\al_s} \ot \pi \ot \al_t$ belong to $\cC_1$, or we have $s \neq \Xi(\pi) t$ so that none of the subrepresentations of $\overline{\al_s} \ot \pi \ot \al_t$ belong to $\cC_1$. Noting that
$$\sum_{\gamma \in \Irr(\cC)} \; d(\gamma) \, \mult(\gamma,\overline{\al_s} \ot \pi \ot \al_t) = d(\al_s) \, d(\pi) \, d(\al_t) \; ,$$
we get that $\psi(\pi) = 1$ for all $\pi \in \Irr(\cC)$. Since $\cC$ has property~(T), it follows that $\psi_n \recht 1$ uniformly on $\Irr(\cC)$.

Take $\pi \in \Irr(\cC_1)$. If $s \neq t$, none of the subrepresentations of $\overline{\al_s} \ot \pi \ot \al_t$ belong to $\cC_1$. If $s=t=e$, then $\overline{\al_s} \ot \pi \ot \al_t = \pi$. We conclude that for all $\pi \in \Irr(\cC_1)$,
\begin{align*}
\psi_n(\pi) &= \frac{1}{|\Lambda|} \, \vphi_{1,n}(\pi) + \vphi_{2,n}(\pi) \quad\text{with}\\ \vphi_{2,n}(\pi) &= \frac{1}{|\Lambda|} \sum_{t \in \Lambda \setminus \{e\}, \gamma \in \Irr(\cC_1)} \frac{\vphi_{1,n}(\gamma) \, d(\gamma) \, \mult(\gamma,\overline{\al_t} \ot \pi \ot \al_t)}{d(\al_t)^2 \, d(\pi)} \; .
\end{align*}
Because $\vphi_{1,n}$ is a cp-multiplier, we have for every $\gamma \in \Irr(\cC_1)$ that $|\vphi_{1,n}(\gamma)| \leq \|\theta^{\vphi_{1,n}}\|\cb = \vphi_{1,n}(\eps) = 1$. Therefore,
$$\bigl|\vphi_{2,n}(\pi)\bigr| \leq \frac{|\Lambda|-1}{|\Lambda|}$$
for all $n$ and all $\pi \in \Irr(\cC_1)$. Since also $|\vphi_{1,n}(\pi)| \leq 1$ for all $\pi \in \Irr(\cC_1)$ and since $\psi_n \recht 1$ uniformly, we conclude from the above that $\vphi_{1,n} \recht 1$ uniformly on $\Irr(\cC_1)$.

Conversely, assume that $\cC_1$ has property~(T). Let $\vphi_n : \Irr(\cC) \recht \C$ be a net of cp-multipliers on $\cC$ that converges to $1$ pointwise. We have to prove that $\vphi_n \recht 1$ uniformly on $\cC$. We may assume that $\vphi_n(\eps) = 1$ for all $n$. Restricting $\vphi_n$ to $\Irr(\cC_1)$, it follows from Proposition \ref{prop.subcat} and the property~(T) of $\cC_1$ that $\vphi_n \recht 1$ uniformly on $\Irr(\cC_1)$. Consider the states $\om_n := \om_{\vphi_n}$ on the C$^*$-algebra $C_u(\cC)$. For every $\al \in \cC$, define $A(\al) = d(\al)^{-1} \al$ as an element of the C$^*$-algebra $C_u(\cC)$. Note that $\|A(\al)\| \leq 1$ for all $\al \in \cC$. Also note that $A(\al)^* = A(\albar)$ and $A(\al) \, A(\be) = A(\al \ot \be)$ for all $\al,\be \in \cC$. Because $\vphi_n \recht 1$ pointwise, we have $\om_n(A(\al)) \recht 1$ for all $\al \in \cC$. Therefore,
$$\om_n\bigl((A(\al) - 1)^* (A(\al) - 1)\bigr) \recht 0 \quad\text{for every}\;\; \al \in \cC \; .$$
By the Cauchy-Schwarz inequality, and using the norm in $C_u(\cC)^*$, we find that
$$\|\om_n \cdot A(\al)^* - \om_n\| \recht 0 \quad\text{for every}\;\; \al \in \cC \; .$$
Since $\Lambda$ is finite and $\|A(\pi)\| \leq 1$ for all $\pi \in \cC$, it follows that
\begin{equation}\label{eq.wehavethis}
|\om_n(A(\al_{\Xi(\pi)})^* A(\pi)) - \om_n(A(\pi))| \recht 0 \quad\text{uniformly in}\;\; \pi \in \Irr(\cC) \; .
\end{equation}
For every $\be \in \cC_1$, we have
$$\om_n(A(\be)) = \frac{1}{d(\be)} \sum_{\gamma \in \Irr(\cC_1)} \, \vphi_n(\gamma) \, d(\gamma) \, \mult(\gamma,\be) \; .$$
Since $\vphi_n \recht 1$ uniformly on $\Irr(\cC_1)$, it follows that $\om_n(A(\be)) \recht 1$ uniformly on all $\be \in \cC_1$. For every $\pi \in \Irr(\cC)$, we have that $\overline{\al_{\Xi(\pi)}} \, \pi$ belongs to $\cC_1$. In combination with \eqref{eq.wehavethis}, we find that $\vphi_n(\pi) = \om_n(A(\pi))$ converges to $1$ uniformly on all $\pi \in \Irr(\cC)$.
\end{proof}

\section{Multipliers on representation categories}\label{sec.mult-rep-cat}

Let $\bG$ be a compact quantum group in the sense of \cite{Wo95}. So, we are given a unital Hopf $*$-algebra $(\cO(\bG),\Delta)$ together with a Haar state $h : \cO(\bG) \recht \C$ satisfying
$$(\id \ot h)\Delta(x) = h(x) 1 = (h \ot \id)\Delta(x) \quad\text{for all}\;\; x \in \cO(\bG)$$
and $h(x^* x) \geq 0$, $h(1) = 1$ for all $x \in \cO(\bG)$. We denote by $C_u(\bG)$ the universal enveloping C$^*$-algebra of $\cO(\bG)$.

Denote by $\Rep(\bG)$ the category of finite-dimensional unitary representations of $\bG$ and by $\Irr(\bG)$ the set of equivalence classes of irreducible unitary representations. For every $\pi \in \Irr(\bG)$, we choose a representative as a unitary matrix $U^\pi \in M_{\dim \pi}(\C) \ot \cO(\bG)$. The matrix coefficients $U^\pi_{ij}$, $\pi \in \Irr(\bG)$, $i,j \in \{1,\ldots,\dim \pi\}$, form a vector space basis for the $*$-algebra $\cO(\bG)$.

In \cite[Definition 1]{DFY13}, a functional $\Om : \cO(\bG) \recht \C$ is called \emph{central} if $(\Om \ot \psi) \circ \Delta = (\psi \ot \Om) \circ \Delta$ for every functional $\psi : \cO(\bG) \recht \C$. As explained in \cite[Section 2]{DFY13}, the central functionals on $\cO(\bG)$ are in bijective correspondence with the functions $\vphi : \Irr(\bG) \recht \C$, by defining
$$\Om_\vphi : \cO(\bG) \recht \C : \Om_\vphi(U^\pi_{ij}) = \vphi(\pi) \, \delta_{ij} \quad\text{for all}\;\; \pi \in \Irr(\bG) \; , \;  i,j \in \{1,\ldots,\dim \pi\} \; .$$
We also associate with $\vphi$ the multiplier
\begin{equation}\label{eq.multiplier-quantum}
\Psi_\vphi : \cO(\bG) \recht \cO(\bG) : \Psi_\vphi = (\id \ot \Om_\vphi)\circ \Delta = (\Om_\vphi \ot \id) \circ \Delta \; .
\end{equation}
Note that $\Psi_\vphi(U^\pi_{ij}) = \vphi(\pi) \, U^\pi_{ij}$ for all $\pi \in \Irr(\bG)$, $i,j \in \{1,\ldots,\dim \pi\}$.

Denote by $L^2(\bG)$ the GNS Hilbert space associated with the Haar state $h$. The reduced C$^*$-algebra $C_r(\bG)$ is defined as the norm closure of $\cO(\bG)$ acting on $L^2(\bG)$. We call $\Psi_\vphi$ a \emph{cb-multiplier} if $\Psi_\vphi$ is completely bounded from $C_r(\bG)$ to $C_r(\bG)$. We say that $\Psi_\vphi$ is finitely supported if $\vphi$ is a finitely supported function.

Given a function $\vphi : \Irr(\bG) \recht \C$, we denote by $\theta^\vphi$ the associated multiplier on $\Rep(\bG)$ given in Proposition \ref{prop.descr-multi}.

The following result is crucial for us, since it allows to transfer the quantum group results in \cite{DFY13,Ar14} to our framework of subfactors and C$^*$-tensor categories.

\begin{proposition}\label{prop.quantum}
Let $\vphi : \Irr(\bG) \recht \C$ a function. The following conditions are equivalent.
\begin{enumlist}
\item The central functional $\Om_\vphi$ on $\cO(\bG)$ is positive: $\Om_\vphi(x^*x) \geq 0$ for all $x \in \cO(\bG)$.
\item The multiplier $\Psi_\vphi : C_r(\bG) \recht C_r(\bG)$ is completely positive.
\item The function $\vphi : \Irr(\bG) \recht \C$ is a cp-multiplier on the C$^*$-tensor category $\Rep(\bG)$.
\end{enumlist}
Also, if $\Psi_\vphi$ is completely bounded, then $\theta^\vphi$ is completely bounded with $\|\theta^\vphi\|\cb \leq \|\Psi_\vphi\|\cb$.
\end{proposition}

\begin{proof}
We start by giving an alternative formula for the maps $\theta^\vphi_{\al,\be}$.
Every $\al \in \Rep(\bG)$ is given as the unitary representation $U^\al \in B(\cH_\al) \ot \cO(\bG)$ of $\bG$ on the finite dimensional Hilbert space $\cH_\al$. By definition,
$$(\id \ot \Delta)(U^\al) = U^\al_{12} \, U^\al_{13} \; .$$
The tensor product of $\al,\be \in \Rep(\bG)$ is the unitary representation
$$U^{\al \ot \be} = U^\al_{13} \; U^\be_{23} \in B(\cH_\al \ot \cH_\be) \ot \cO(\bG) \; .$$
Fix $\al,\be \in \Rep(\bG)$. Define the linear map
$$\sigma^\vphi_{\al,\be} : \End(\al \ot \be) \recht B(\cH_\al \ot \cH_\be) : \sigma^\vphi_{\al,\be}(X) = (\id \ot \id \ot \Om_\vphi)(U^\be_{23} (X \ot 1) (U^\be_{23})^*) \; .$$
We start by proving that $\sigma^\vphi_{\al,\be}(X) = \theta^\vphi_{\al,\be}(X)$ for all $X \in \End(\al \ot \be)$.

Let $\bebar \in \Rep(\bG)$ be the contragredient unitary representation. Fix a standard solution $s_\be \in \Mor(\be \ot \bebar,\eps)$, $t_\be \in \Mor(\bebar \ot \be,\eps)$ of the conjugate equations as in \eqref{eq.standard-sol}. Fix $X \in \End(\al \ot \be)$. To prove that $\sigma^\vphi_{\al,\be}(X) = \theta^\vphi_{\al,\be}(X)$, we must show that
\begin{equation}\label{eq.goal}
(\sigma^\vphi_{\al,\be}(X) \ot 1)(1 \ot s_\be) = (\theta^\vphi_{\al,\be}(X) \ot 1) (1 \ot s_\be) \; .
\end{equation}
As in the proof of Proposition \ref{prop.descr-multi}, we denote for every $\pi \in \Irr(\bG)$, by $P^{\be \ot \bebar}_\pi$ the orthogonal projection in $\End(\be \ot \bebar)$ onto the direct sum of all subrepresentations of $\be \ot \bebar$ that are isomorphic with $\pi$. By \eqref{eq.second-def}, the right hand side of \eqref{eq.goal} equals
\begin{equation}\label{eq.rhs}
\sum_{\pi \in \Irr(\bG)} \vphi(\pi) \; (1 \ot P^{\be \ot \bebar}_\pi)(X \ot 1)(1 \ot s_\be) \; .
\end{equation}
Since $s_\be \in \Mor(\be \ot \bebar,\eps)$, we have $U^\be_{13} \, U^{\bebar}_{23} \, (s_\be \ot 1) = s_\be \ot 1$ and thus, $(U^\be_{13})^* (s_\be \ot 1) = U^{\bebar}_{23}(s_\be \ot 1)$. It follows that the left hand side of \eqref{eq.goal} equals
\begin{multline*}(\id \ot \id \ot \id \ot \Om_\vphi)\bigl( U^\be_{24} \, (X \ot 1 \ot 1) \, U^{\bebar}_{34} \, (1 \ot s_\be \ot 1) \bigr) \\ =
\bigl( 1 \ot (\id \ot \id \ot \Om_\vphi)(U^{\be \ot \bebar}) \bigr) (X \ot 1) (1 \ot s_\be) \; .
\end{multline*}
By definition of $\Om_\vphi$, we have that $(\id \ot \Om_\vphi)(U^\pi) = \vphi(\pi) 1$ for all $\pi \in \Irr(\bG)$. Therefore,
$$(\id \ot \id \ot \Om_\vphi)(U^{\be \ot \bebar}) = \sum_{\pi \in \Irr(\bG)} \vphi(\pi) \, P^{\be \ot \bebar}_\pi$$
and we conclude that also the left hand side of \eqref{eq.goal} equals \eqref{eq.rhs}.

Define the injective $*$-homomorphism $\Gamma_{\al,\be} : \End(\al \ot \be) \recht B(\cH_\al \ot \cH_\be) \ot C_r(\bG) : \Gamma_{\al,\be}(X) = U^\be_{23} (X \ot 1) (U^\be_{23})^*$. We next prove that
\begin{equation}\label{eq.cruciallink}
(\id \ot \id \ot \Psi_\vphi) \circ \Gamma_{\al,\be} = \Gamma_{\al,\be} \circ \theta^\vphi_{\al,\be} \; .
\end{equation}
To prove \eqref{eq.cruciallink}, fix $X \in \End(\al \ot \be)$. Using the equality $\sigma^\vphi_{\al,\be} = \theta^\vphi_{\al,\be}$, we get that
\begin{align*}
(\id \ot \id \ot \Psi_\vphi)\Gamma_{\al,\be}(X) & = (\id \ot \id \ot \id \ot \Om_\vphi)(\id \ot \id \ot \Delta) \Gamma_{\al,\be}(X) \\
& = (\id \ot \id \ot \id \ot \Om_\vphi)\bigl( U^\be_{23} \, U^\be_{24} \, (X \ot 1 \ot 1) \, (U^\be_{24})^* \, (U^\be_{23})^* \bigr) \\
& = U^\be_{23} \, (\sigma^\vphi_{\al,\be}(X) \ot 1) \, (U^\be_{23})^* = U^\be_{23} \, (\theta^\vphi_{\al,\be}(X) \ot 1) \, (U^\be_{23})^* \\
& = \Gamma_{\al,\be}(\theta^\vphi_{\al,\be}(X)) \; .
\end{align*}

From \eqref{eq.cruciallink}, it immediately follows that if $\Psi_\vphi$ is completely bounded, then every $\theta^\vphi_{\al,\be}$ is completely bounded with $\|\theta_{\al,\be}^\vphi\|\cb \leq \|\Psi_\vphi\|\cb$.

We now prove that statements 1, 2 and 3 are equivalent. Since $\Psi_\vphi = (\id \ot \Om_\vphi) \circ \Delta$, we get that 1~$\Rightarrow$~2. From \eqref{eq.cruciallink}, it follows that 2~$\Rightarrow$~3.

It remains to prove that 3~$\Rightarrow$~1. So assume that $\theta^\vphi$ is completely positive and fix $x \in \cO(\bG)$. We have to prove that $\Om_\vphi(xx^*) \geq 0$. Take $\be \in \Rep(\bG)$ such that $x$ is a coefficient of the representation $\be$. Choosing standard solutions $s_\be, t_\be$ of the conjugate equations as in \eqref{eq.standard-sol}, we can take a vector $\xi \in \cH_{\bebar} \ot \cH_\be$ such that
$$x = (\xi^* \ot 1) \, U^\be_{23} \, (t_\be \ot 1) \; .$$
Since $\theta^\vphi_{\bebar,\be} = \sigma^\vphi_{\bebar,\be}$ and because $\theta^\vphi_{\bebar,\be}$ is a positive map, we get that
$$\Om_\vphi(xx^*) = \langle \sigma^\vphi_{\bebar,\be}(t_\be t_\be^*) \, \xi , \xi \rangle= \langle \theta^\vphi_{\bebar,\be}(t_\be t_\be^*) \, \xi , \xi \rangle \geq 0 \; .$$
This concludes the proof of the proposition.
\end{proof}

Proposition \ref{prop.quantum} gives a description of the cp-multipliers on the representation category $\Rep(\bG)$ of a compact quantum group $\bG$. We can then apply \cite[Section 6]{DFY13} to identify the universal C$^*$-algebra $C_u(\Rep(\bG))$ of the C$^*$-tensor category $\Rep(\bG)$ as a corner of the universal \emph{quantum double} C$^*$-algebra of $\bG$.

To formulate the result, we need some notation. First define $c_c(\bGhat)$ as the direct sum of the matrix algebras $B(\cH_\pi)$, $\pi \in \Irr(\bG)$. We can naturally embed $c_c(\bGhat) \hookrightarrow \cO(\bG)^*$ by identifying the matrix units $e^\pi_{ij} \in B(\cH_\pi)$ with the functionals $e^\pi_{ij} \in \cO(\bG)^*$ given by
$$e^\pi_{ij}(U^\rho_{kl}) = \delta_{\pi,\rho} \, \delta_{jk} \, \delta_{il} \; .$$
The \emph{Drinfel'd double} $\cD(\bG)$ of $\bG$ is defined as the $*$-algebra with underlying vector space $\cO(\bG) \ot c_c(\bGhat)$ and product
$$(U^\pi_{ij} \ot \om) \, (U^\rho_{kl} \ot \mu) = \sum_{a,b} U^\pi_{ij} \, U^\rho_{ab} \ot \om\bigl( U^\rho_{ka} \, \cdot \, (U^\rho_{lb})^* \bigr) \, \mu$$
for all $\pi,\rho \in \Irr(\bG)$ and $\om,\mu \in c_c(\bGhat)$ viewed as functionals on $\cO(\bG)$. As explained in \cite[Section 6]{DFY13}, the (non-unital) $*$-algebra $\cD(\bG)$ admits a universal enveloping C$^*$-algebra $D(\bG)$. From now on, we write $x \om$ instead of $x \ot \om$ as elements of $\cD(\bG) \subset D(\bG)$. The Haar state $h \in \cO(\bG)$ corresponds to the minimal projection in $c_c(\bGhat)$ given by the trivial representation $\eps$ and as such also is a projection in $D(\bG)$.

\begin{proposition}\label{prop.univ-Cstar-rep-category}
Let $\bG$ be a compact quantum group. The formula
$$\Phi : \C[\Rep \bG] \recht h D(\bG) h : \Phi(\pi) = \sigma_{-i/2}(\chi_\pi) \, h$$
where $\chi_\pi = (\Tr \ot \id)(U^\pi)$ is the character of $\pi$ and $(\sigma_t)_{t \in \R}$ is the modular automorphism group of the Haar state $h$ induces a bijective $*$-isomorphism
$$\Phi : C_u(\Rep \bG) \recht h D(\bG) h \; .$$
\end{proposition}
\begin{proof}
Let $\vphi : \Irr(\bG) \recht \C$ be a function. In \cite[Section 6]{DFY13}, it is proved that the central functional $\Om_\vphi$ on $\cO(\bG)$ is positive if and only if there exists a positive functional $\Om'_\vphi$ on the C$^*$-algebra $h D(\bG) h$ such that $\Om'_\vphi(\Phi(\pi)) = d(\pi) \, \vphi(\pi)$ for all $\pi \in \Irr(\bG)$. Also, $\Phi(\C[\Rep \bG]) = h \cD(\bG) h$ and thus, $\Phi(\C[\Rep \bG])$ is dense in $h D(\bG) h$. The result then follows from Proposition \ref{prop.quantum}.
\end{proof}

Given a compact quantum group $\bG$, we denote by $C_u(\bG)$ the universal enveloping C$^*$-algebra of $\cO(\bG)$. We denote by $\counit : \cO(\bG) \recht \C$ the functional given by $\counit(U^\pi_{ij}) = \delta_{ij}$ for all $\pi \in \Irr(\bG)$. Note that $\counit$ is exactly the positive central functional on $\cO(\bG)$ that corresponds to the function on $\Irr(\bG)$ that is identically equal to $1$. Recall from \cite[Definition 8.1]{Ar14} that $\bGhat$ is said to have the \emph{central property~\pT} if every net of central states $\om_n \in C_u(\bG)^*$ that converges to $\counit$ weakly$^*$, must satisfy $\lim_n \|\om_n - \counit\| = 0$ w.r.t.\ the norm on $C_u(\bG)^*$.

\begin{proposition}\label{prop.equiv-prop-T}
Let $\bG$ be a compact quantum group. Then $\bGhat$ has the central property~\pT\ in the sense of \cite[Definition 8.1]{Ar14} if and only if the rigid C$^*$-tensor category $\Rep(\bG)$ has property~\pT\ in the sense of Definition \ref{def.prop-Cstar-cat}.
\end{proposition}

\begin{proof}
First assume that $\bGhat$ has the central property~(T). Let $\vphi_n : \Irr(\bG) \recht \C$ be a net of cp-multipliers that converges to $1$ pointwise. We have to prove that $\vphi_n \recht 1$ uniformly on $\Irr(\bG)$. Replacing $\vphi_n$ by $\vphi_n(\eps)^{-1} \, \vphi_n$, we may assume that $\vphi_n(\eps) = 1$ for all $n$. Denote by $\Om_n := \Om_{\vphi_n}$ the central functionals on $\cO(\bG)$ given by $\Om_n(U^\pi_{ij}) = \vphi_n(\pi) \delta_{ij}$ for all $\pi \in \Irr(\bG)$. By Proposition \ref{prop.quantum}, the functionals $\Om_n$ uniquely extend to states on $C_u(\bG)$ that we still denote as $\Om_n$. Since $\vphi_n \recht 1$ pointwise, we get that $\Om_n \recht \counit$ weakly$^*$. Since $\bGhat$ has the central property~(T), we get that $\lim_n \|\Om_n - \counit\| = 0$. For every $\pi \in \Irr(\bG)$, we have that $\|U^\pi_{11}\| \leq 1$. Therefore,
$$|\vphi_n(\pi) - 1| = |\Omega_n(U^\pi_{11}) - \counit(U^\pi_{11})| \leq \|\Omega_n - \counit\| \;\;,$$
and we conclude that $\vphi_n \recht 1$ uniformly.

Conversely, assume that $\Rep(\bG)$ has property~(T) in the sense of Definition \ref{def.prop-Cstar-cat}. Take a net $\Om_n \in C_u(\bG)^*$ of central states that converges weakly$^*$ to $\counit$. Using Proposition \ref{prop.quantum}, denote by $\vphi_n : \Irr(\bG) \recht \C$ the cp-multipliers on $\Rep (\bG)$ that correspond to $\Om_n$. Denote by $\om_n \in C_u(\Rep (\bG))^*$ the states defined by $\vphi_n$. Note that $\vphi_n \recht 1$ pointwise and $\omega_n \recht \counit$ weakly$^*$. By Proposition \ref{prop.char-prop-T}, we get that $\lim_n \|\om_n - \counit\| = 0$. We now use the notation introduced before Proposition \ref{prop.univ-Cstar-rep-category}. With the $*$-isomorphism $\Phi$ of Proposition \ref{prop.univ-Cstar-rep-category}, we can define the states $\Omtil_n, \counittil$ on $D(\bG)$ given by $\om_n \circ \Phi^{-1}$ and $\counit \circ \Phi^{-1}$. We have $\lim_n \|\Omtil_n - \counittil\| = 0$. Using the natural $*$-homomorphism $C_u(\bG) \recht M(D(\bG))$, we have that $\Om_n(x) = \Omtil_n(h x h)$ for all $x \in C_u(\bG)$. It then follows that $\lim_n \|\Om_n - \counit\| = 0$.
\end{proof}

\section[\boldmath Representation theory and approximation properties of TLJ $\lambda$-lattices]{\boldmath Representation theory and approximation properties of \linebreak TLJ $\lambda$-lattices} \label{sec.rep-theory-TLJ}

As explained in the introduction, the TLJ $\lambda$-lattice $\cG^\lambda$ arises as the sub-$\lambda$-lattice generated by the Jones projections inside the standard invariant of any extremal subfactor of index $\lambda^{-1}$. By \cite{Po90}, there exist (non hyperfinite) subfactors $N \subset M$ whose standard invariant is generated by the Jones projections and thus equal to $\cG^\lambda$. These are precisely the subfactors whose principal graph is of type $A_n$ or $A_\infty$. The finite depth case $A_n$ arises when $\lambda^{-1} < 4$, while for $\lambda^{-1} \geq 4$, we have infinite depth and principal graph $A_\infty$.

The C$^*$-tensor category corresponding to $\cG^\lambda$ with $\lambda^{-1} \geq 4$ can be described as follows. Assume that $N \subset M$ is a subfactor with index $\lambda^{-1} \geq 4$ and principal graph $A_\infty$. Consider the Jones tower $N \subset M \subset M_1 \subset \cdots$. We define $\cH_0 = L^2(M)$ to be the trivial $M$-bimodule. For every $n \geq 1$, the $M$-bimodule $L^2(M_n)$ is the direct sum of irreducible $M$-bimodules that appear in $L^2(M_{n-1})$, plus one new irreducible $M$-bimodule that we denote as $\cH_n$. Then $\{\cH_n \mid n \geq 0\}$ is the set of irreducible $M$-bimodules generated by $N \subset M$.

\begin{theorem}\label{thm.main-temperley-lieb}
Let $N \subset M$ be an $A_\infty$ subfactor with index $[M:N]=\lambda^{-1}$ and standard invariant $\cG^\lambda$. Denote by $\cC$ the associated category of $M$-bimodules and identify $\Irr(\cC) = \{\cH_n \mid n \in \N\}$ as above.
\begin{enumlist}
\item There is a unique bijective $*$-isomorphism $\Theta : C_u(N \subset M) \recht C([0,\lambda^{-1}])$ that maps the $M$-bimodule $L^2(M_1)$ to the function $t \mapsto t$.

\item The irreducible \SE-correspondences of $N \subset M$ are labeled by the interval $[0,\lambda^{-1}]$ and explicitly given by the cp-multipliers $\vphi_t : \Irr(\cC) \recht \C$, $t \in [0:\lambda^{-1}]$, defined by
    $$\vphi_t(\cH_n) = \frac{V_n(t)}{V_n([M:N])} \quad\text{where}\;\; V_n(t) = U_{2n}\bigl(\frac{1}{2} \sqrt{t}\bigr)$$
    and the $U_m$ are the Chebyshev polynomials of the second kind.

\item The regular \SE-correspondence of $N \subset M$ is given by the representation of $C([0,\lambda^{-1}])$ on $L^2([0,4])$ by left multiplication. The trivial \SE-correspondence is given by evaluation at $\lambda^{-1}$.

\item The TLJ $\lambda$-lattice $\cG^\lambda$ has the Haagerup approximation property and CMAP.
\end{enumlist}
\end{theorem}

\begin{proof}
Take the unique number $0 < q \leq 1$ with $q+\frac{1}{q} = \lambda^{-1/2}$. Denote by $\bG$ the compact quantum group $\SU_q(2)$ in the sense of \cite{Wo86}. Denote by $\al_{1/2}$ the defining dimension $2$ (irreducible) representation of $\bG$. Then $\cC$ is equivalent, as a C$^*$-tensor category, with the full C$^*$-tensor subcategory of $\Rep(\bG)$ generated by $\al_{1/2} \ot \al_{1/2}$. Combining Proposition \ref{prop.Cstar-subcat} with Proposition \ref{prop.univ-Cstar-rep-category} and \cite[Remark 31]{DFY13}, we find a unique injective $*$-isomorphism $C_u(N \subset M) \hookrightarrow C([-\lambda^{-1/2},\lambda^{-1/2}])$ sending the $M$-bimodule $L^2(M_1)$ to the function $t \mapsto t^2$. Then the first statement of the theorem follows.

Define the functions $V_n(t)$ as in 2. It is easy to check that the functions $V_n(t)$ are determined by the recursive relation
\begin{equation}\label{eq.recurrence-Vn}
V_0(t) = 1 \;\; , \;\; V_1(t) = t-1 \quad\text{and}\quad V_{n+1}(t) = (t-2)V_n(t) - V_{n-1}(t) \;\;\quad\text{for all}\;\; n \geq 1 \; .
\end{equation}
In particular, all $V_n(t)$ are polynomials in $t$. The fusion algebra $\C[\cC]$ can be described by the fusion rule
$$\cH_1 \ot_M \cH_n \cong \cH_{n+1} \oplus \cH_n \oplus \cH_{n-1} \quad\text{for all}\;\; n \geq 1 \; .$$
So in the fusion algebra $\C[\cC]$, we have that $\cH_n = V_n(L^2(M_1))$ and that $d(\cH_n) = V_n(\lambda^{-1})$. Since the irreducible representations of $C([0,\lambda^{-1}])$ are obviously given by evaluation at $t \in [0,\lambda^{-1}]$, the second statement follows.

Since $\vphi_{\lambda^{-1}}(\cH_n) = 1$ for all $n \geq 0$, we get that the trivial \SE-correspondence of $N \subset M$ is indeed given by evaluation at $\lambda^{-1}$. Defining the probability measure $\mu$ with support $[0,4]$ and density $(2\pi)^{-1/2} \sqrt{(4-t)/t}$, the orthogonality relations for the Chebyshev polynomials of the second kind imply that $\int_0^4 V_n(t) \, d\mu(t) = 0$ for all $n \geq 1$. This means that $\int_0^4 \vphi_t \, d\mu(t) = \vphi_0$, so that indeed the regular \SE-correspondence of $N \subset M$ is given by the left multiplication representation of $C([0,\lambda^{-1}])$ on $L^2([0,4])$.

It is easy to check that for every fixed $0 < t < \lambda^{-1}$, we have that $\lim_{n \recht \infty} \vphi_t(\cH_n) = 0$. Therefore, $\cC$ has the Haagerup approximation property.

By \cite[Theorem 16]{DFY13}, there exists a sequence of finitely supported functions $\vphi_n : \Irr(\bG) \recht \C$ that converges to $1$ pointwise and such that the associated multipliers $\Psi_{\vphi_n} : C_r(\bG) \recht C_r(\bG)$ given by \eqref{eq.multiplier-quantum} satisfy $\limsup_n \|\Psi_{\vphi_n}\|\cb = 1$. By Propositions \ref{prop.quantum} and \ref{prop.subcat}, the restriction of $\vphi_n$ to $\Irr(\cC) \subset \Irr(\bG)$ is a sequence of finitely supported cb-multipliers $\psi_n$ on $\cC$ that converges to $1$ pointwise and that satisfies $\lim_n \|\psi_n\|\cb = 1$. So, $\cC$ has CMAP.
\end{proof}

\section{\boldmath Property~(T) for $\lambda$-lattices of type $\SU_q(n)$, $n \geq 3$}

For every $n \geq 2$, consider the compact quantum group $\bG = \SU_q(n)$ of \cite{Wo88}, with its defining dimension $n$ (irreducible) unitary representation $U^\pi \in M_n(\C) \ot \cO(\bG)$. We can then define $\bG_1 = \PSU_q(n)$ such that $\cO(\bG_1)$ is the Hopf $*$-subalgebra of $\cO(\bG)$ generated by the coefficients of $\pi \ot \pibar$.

Since the fusion rules of $\bG$ are the same as the fusion rules of the compact Lie group $\SU(n)$, we can also view $\cO(\bG_1)$ in the following way. There is a unique map $\Xi : \Irr(\bG) \recht \Z/n\Z$ satisfying $\Xi(\pi) = 1$ and $\Xi(\al) = \Xi(\be) +\Xi(\gamma)$ whenever $\al,\be,\gamma \in \Irr(\bG)$ and $\al$ is a subrepresentation of $\be \ot \gamma$.

\begin{theorem}\label{thm.main-SUq3}
Let $n \geq 3$ be an odd integer. Let $N \subset M$ be an extremal finite index subfactor and assume that the associated category of $M$-bimodules is equivalent with the representation category $\Rep(\SU_q(n))$ or $\Rep(\PSU_q(n))$. Then, the standard invariant of $N \subset M$ has property~\pT.
\end{theorem}

By \cite{Po94b,Xu97,Ba98} (see Remark \ref{rem.exist} below), subfactors $N \subset M$ whose associated category of $M$-bimodules is equivalent with $\Rep(\SU_q(n))$ or $\Rep(\PSU_q(n))$ indeed exist. In the case of $\PSU_q(n)$, they can be chosen irreducible and without intermediate subfactors.

The subfactors $N \subset M$ in Theorem \ref{thm.main-SUq3} are the first subfactors that have a standard invariant with property~(T) and that are not constructed from discrete groups with property~(T). Note that the fusion algebra associated with these subfactors $N \subset M$ is abelian and thus not at all like $C^*(\Gamma)$ for a property~(T) group $\Gamma$.

\begin{proof}
Write $\bG = \SU_q(n)$ and $\bG_1 = \PSU_q(n)$. By \cite[Corollary 8.8]{Ar14}, the discrete quantum group $\bGhat$ has the central property~(T). So by Proposition \ref{prop.equiv-prop-T}, the category $\Rep(\bG)$ has property~(T). By Proposition \ref{prop.stability-T}, also $\Rep(\bG_1)$ has property~(T). So, the category of $M$-bimodules generated by $N \subset M$ has property~(T). It then follows from Proposition \ref{prop.all-equiv} that the standard invariant of $N \subset M$ has property~(T).
\end{proof}

\begin{remark}\label{rem.exist}
By \cite{Xu97,Ba98}, we can associate a $\lambda$-lattice to every rigid C$^*$-tensor category $\cC$ and object $\al \in \cC$, in the following way. Fix a standard solution $s_\al \in \Mor(\al \ot \albar,\eps)$, $t_\al \in \Mor(\albar \ot \al,\eps)$ of the conjugate equations as in \eqref{eq.standard-sol}. For every $k \in \N$, define $\al_k = \al$ if $k$ is odd and $\al_k = \albar$ if $k$ is even. Then define for all $n \leq m$,
$$\al_{nm} = \al_{n+1} \ot \cdots \ot \al_m$$
with the convention that $\al_{nn} = \eps$. Denote $A_{nm} = \End(\al_{nm})$. For $a \leq n \leq m \leq b$, we have the canonical inclusion $A_{nm} \subset A_{ab}$. We also have the projections $e_n \in A_{n-1,n+1}$ defined by $d(\al)^{-1} s_\al s_\al^*$ if $n$ is odd and by $d(\al)^{-1} t_\al t_\al^*$ if $n$ is even. Together with the normalized categorical trace and $\lambda = d(\al)^{-2}$, we find that $A_{nm}$ is a $\lambda$-lattice. By \cite[Theorem 3.1]{Po94b}, this $\lambda$-lattice $A_{nm}$ arises as the standard invariant of an extremal finite index subfactor $N \subset M$,
constructed in a canonical way from the $\lambda$-lattice and a diffuse tracial von Neumann algebra $Q$. By \cite[Theorem 1.1]{PS01}, by taking $Q$ to be the free group factor
on infinitely many generators $L(\mathbb F_\infty)$,
one associate this way a canonical inclusion $N\subset M$ with $N \cong M \cong L(\F_\infty)$. Note that $N \subset M$ is irreducible if and only if $\al$ is an irreducible object in $\cC$.

By construction, the category $\cC_1$ of $M$-bimodules generated by such a subfactor $N \subset M$ is equivalent with the full C$^*$-tensor subcategory of $\cC$ generated by $\al \ot \albar$. If we take $\cC = \Rep(\SU_q(n))$ with the defining $n$-dimensional representation denoted by $\pi$, we could take $\al = \pi \oplus \eps$ and find that $\cC_1 = \cC$. Then the subfactor $N \subset M$ is reducible. We could also take $\al = \pi$. Then, $\cC_1 = \Rep(\PSU_q(n))$ and the subfactor $N \subset M$ is irreducible. In that case, because $\pibar \ot \pi$ is the direct sum of the trivial representation and an irreducible representation, and because by construction $N' \cap M_1 \cong \End(\pibar \ot \pi)$, we find that $N' \cap M_1 = \C 1 + \C e_0$. In particular, $N \subset M$ has no intermediate subfactor.
\end{remark}

\section{\boldmath Permanence properties and Fuss-Catalan $\lambda$-lattices}\label{sec.permanence}

We collect in this section a number of results on the permanence under various constructions of the Haagerup property, weak amenability, property~(T),~.... We formulate most of these permanence properties for $\lambda$-lattices. They have their obvious counterpart for rigid C$^*$-tensor categories, but having the SE-inclusion at hand makes some of the proofs much less laborious. Moreover, since every finitely generated rigid C$^*$-tensor category is the category of bimodules of an extremal subfactor (see also Remark~\ref{rem.exist}), there is not even a loss of generality.

To obtain a permanence result for free products, we need the following approximation property, which was defined in \cite[Definition 3]{DFY13} for discrete quantum groups.

\begin{definition}\label{def.ACPAP}
A rigid C$^*$-tensor category $\cC$ is said to have the \emph{almost} completely positive approximation property (ACPAP) if there exists a net of cp-multipliers $\vphi_n : \Irr(\cC) \recht \C$ that converges to $1$ pointwise and such that for every fixed $n$, there exists a net of finitely supported cb-multipliers $\psi^n_k : \Irr(\cC) \recht \C$ such that $\lim_k \|\vphi_n - \psi^n_k\|\cb = 0$.
\end{definition}

Note that the ACPAP is stronger than both the Haagerup property and CMAP.

\begin{proposition}\label{prop.TLJ-ACPAP}
The TLJ $\lambda$-lattices $\cG^\lambda$ have the ACPAP.
\end{proposition}
\begin{proof}
By \cite[Theorem 16]{DFY13} and Proposition \ref{prop.quantum}, the representation categories of the compact quantum groups $\SU_q(2)$ and $\PSU_q(2)$ have the ACPAP. As in Section \ref{sec.rep-theory-TLJ}, this precisely means that $\cG^\lambda$ has the ACPAP.
\end{proof}

Note that point 3 in the following proposition corresponds to the permanence of approximation properties when passing to subgroups.

\begin{proposition}\label{prop.permanence}
Let $N \subset M$ be an extremal subfactor with standard invariant $\cG_{N,M}$ and tower/tunnel $\cdots \subset M_{-2} \subset M_{-1} \subset M_0 \subset M_1 \subset \cdots$.
\begin{enumlist}
\item $\cG_{N,M}$ has both the Haagerup property and property~\pT\ if and only if $\cG_{N,M}$ has finite depth.
\item Consider intermediate subfactors $M_a \subset P \subset M_n \subset M_m \subset Q \subset M_b$ for $a \leq n < m \leq b$. Any of the properties of amenability, Haagerup property, property~\pT, having Cowling-Haagerup constant equal to $\kappa$, or ACPAP holds for one of the $\cG_{P,Q}$ if and only if it holds for all $\cG_{P,Q}$.
\item Assume that $M_n \subset P \subset Q \subset M_m$ for some $n < m$ and that $p \in P' \cap Q$ is a projection. If $\cG_{N,M}$ has any of the properties of amenability, Haagerup property, or ACPAP, then also $\cG_{Pp,pQp}$ has the corresponding property. Also, $\Lambda(\cG_{Pp,pQp}) \leq \Lambda(\cG_{N,M})$.
\item Let also $P \subset Q$ be an extremal subfactor. Any of the properties of amenability, Haagerup property, property~\pT, or ACPAP holds for $\cG_{N \ovt P,M \ovt Q}$ if and only if it holds for both $\cG_{N,M}$ and $\cG_{P,Q}$. Also, $\Lambda(\cG_{N \ovt P,M \ovt Q}) = \Lambda(\cG_{N,M}) \, \Lambda(\cG_{P,Q})$.
\item Let $N \subset P \subset M$ and assume that $N \subset M$ is a free composition of $N \subset P$ and $P \subset M$ in the sense of \cite[Section 1]{BJ95}. If both $\cG_{N,P}$ and $\cG_{P,M}$ have the Haagerup property, then also $\cG_{N,M}$ has the Haagerup property. If both $\cG_{N,P}$ and $\cG_{P,M}$ have the ACPAP, then also $\cG_{N,M}$ has the ACPAP.
\end{enumlist}
\end{proposition}
\begin{proof}
1. Since $\cG_{N,M}$ has finite depth if and only if the category of $M$-bimodules generated by $N \subset M$ has only finitely many irreducible objects, the results follows immediately from the definitions.

2 and 4. The proofs are identical to the proof of \cite[Proposition 9.8]{Po99}.

3. Using 2, it suffices to prove that all the stated approximation properties for $\cG_{N,M}$ are inherited by $\cG_{N,P}$ and by $\cG_{P,M}$ whenever $N \subset P \subset M$, as well as by $\cG_{Np,pMp}$ when $p \in N' \cap M$ is a projection. Denote by $\cC$ the category of $M$-bimodules generated by $N \subset M$. Note that the category of $M$-bimodules generated by $P \subset M$ is a full C$^*$-tensor subcategory $\cC_1$ of $\cC$. By Proposition \ref{prop.subcat}, all approximation properties for $\cC$ are inherited by $\cC_1$, i.e.\ by $\cG_{P,M}$. Choosing a tunnel $Q \subset N \subset P$, we can view $M_{-2} \subset Q \subset N$. Again using 2, we get that $\cG_{Q,N}$, and thus also $\cG_{N,P}$ inherit all approximation properties. Using the natural bijective correspondence between $pMp$-bimodules and $M$-bimodules, also the category of $pMp$-bimodules generated by $Np \subset pMp$ can be viewed as a full C$^*$-tensor subcategory of $\cC$.

5. By assumption, the categories $\cC_1$, resp.\ $\cC_2$ of $P$-bimodules generated by $N \subset P$, resp.\ $P \subset M$, are free and generate the category of $P$-bimodules associated with $P \subset M_1$. Put $T = P \ovt P\op$. Denote by $T \subset S_1$ and $T \subset S_2$ the \SE-inclusions of $N \subset P$, resp.\ (a tunnel for) $P \subset M$. By freeness and the categorical description of the symmetric enveloping algebra given in Remark \ref{rem.nonextremal}, it follows that the inclusion of $T$ into the amalgamated free product $S_1 *_T S_2$ is the \SE-inclusion of (a tunnel for) $P \subset M_1$.

First assume that $\cC_1$ and $\cC_2$ have the Haagerup property and choose nets of cp-multipliers $\vphi^{(i)}_n : \Irr(\cC_i) \recht \C$ that converge to $1$ pointwise and such that every $\vphi^{(i)}_n$ tends to $0$ at infinity. We may assume that $\vphi^{(i)}_n(\eps) = 1$ for all $n$ and $i=1,2$. Whenever $i_1,\ldots,i_d \in \{1,2\}$ with $i_1 \neq i_2$, $i_2 \neq i_3, \ldots ,$ $i_{d-1} \neq i_d$, and $\al_k \in \Irr(\cC_{i_k}) \setminus \{\eps\}$, we define
$$\psi_n(\al_1 \cdots \al_d) = r^d \, \vphi^{(i_1)}_n(\al_1) \, \vphi^{(i_2)}_n(\al_2) \, \cdots \, \vphi^{(i_d)}_n(\al_d) \; .$$
Since $\cC$ is the free product of $\cC_1$ and $\cC_2$, we have defined a net of multipliers $\psi_n : \Irr(\cC) \recht \C$. By construction, each $\psi_n$ tends to $0$ at infinity. It follows from Lemma \ref{lem.all-the-same}, Propositions \ref{prop.ext-mult} and \ref{prop.from-si-to-tc}, the amalgamated free product structure of $S_1 *_T S_2$ and \cite[Theorem 3.4 and Proposition 3.5]{RX05} (which are indeed valid in the amalgamated case, as explained in \cite[Section 5]{RX05}) that all $\psi_n$ are completely positive. So $\cC$ has the Haagerup property. This means that $\cG_{P,M_1}$ has the Haagerup property. It now follows from 2 that also $\cG_{N,M}$ has the Haagerup property.

The argument for the ACPAP is entirely similar and worked out in detail in \cite[Proposition 23]{DFY13}. The main point is to use \cite[Lemma 4.10]{RX05} (and also this remains valid in the amalgamated case, as explained in \cite[Section 5]{RX05}).
\end{proof}

\begin{corollary}\label{cor.fuss-catalan}
The Fuss-Catalan $\lambda$-lattices of \cite{BJ95} have both the Haagerup property and CMAP.
\end{corollary}
\begin{proof}
Since the Fuss-Catalan $\lambda$-lattices arise as the free composition of $\cG^{\lambda_i}$, the result follows by combining Propositions \ref{prop.TLJ-ACPAP} and \ref{prop.permanence}.4.
\end{proof}

As is well known, the quotient of a property~(T) group has again property~(T). The counterpart for $\lambda$-lattices was proved in \cite[Theorem 9.9]{Po99} using subtle analytic
arguments in the associated SE-inclusion~: if $\cG_0$ is a sub-$\lambda$-lattice of $\cG$ and if $\cG_0$ has property~(T), then also $\cG$ has property~(T). We end this paper with a
generalization of this type of result to rigid C$^*$-tensor categories, for which our framework will allow us to provide a very simple, completely algebraic proof.

Recall here that a sub-$\lambda$-lattice of $(A_{nm})_{0 \leq  n \leq m}$ is a system of $*$-subalgebras $B_{nm} \subset A_{nm}$, compatible with all inclusions and containing the Jones projections. For instance, $\cG^\lambda$ is a sub-lattice of any $\lambda$-lattice $\cG$.

In the language of rigid C$^*$-tensor categories, this means that we consider a C$^*$-tensor functor $\cF : \cC \recht \cC_1$ with the property that $d_1(\cF(\al)) = d(\al)$ for all $\al \in \cC$, where $d_1, d$ are the categorical dimension functions on $\cC, \cC_1$. Note that the rigidity of $\cC$ and $\cC_1$ implies that $\cF$ is faithful, i.e.\ injective on spaces of morphisms.
As we will recall in the proof of Proposition \ref{prop.quotient-hom} below, the assumption that $\cF$ preserves dimensions is equivalent with the assumption that for every standard solution of the conjugate equations $s_\al \in \Mor(\al \ot \albar,\eps)$, $t_\al \in \Mor(\albar \ot \al,\eps)$, we have that $\cF(s_\al)$, $\cF(t_\al)$ is a standard solution of the conjugate equations for $\cF(\al) \in \cC_1$. In the context of $\lambda$-lattices, this last property exactly means that the sub-$\lambda$-lattice has the same Jones projections as the ambient $\lambda$-lattice. We refer to Remark \ref{rem.essential-preserve-salpha} for further comments on this.

\begin{proposition}\label{prop.quotient-hom}
Let $\cC$, $\cC_1$ be rigid C$^*$-tensor categories and $\cF : \cC \recht \cC_1$ a dimension preserving C$^*$-tensor functor. Then the map $\al \mapsto \cF(\al)$ extends to a unital $*$-homomorphism $C_u(\cC) \recht C_u(\cC_1)$.
\end{proposition}
\begin{proof}
We start by proving that for every standard solution of the conjugate equations $s_\al \in \Mor(\al \ot \albar,\eps)$, $t_\al \in \Mor(\albar \ot \al,\eps)$, we have that $\cF(s_\al)$, $\cF(t_\al)$ is a standard solution of the conjugate equations for $\cF(\al) \in \cC_1$. Since $\cF(s_\al), \cF(t_\al)$ is a solution of the conjugate equations for $\cF(\al)$, we can decompose $\cF(\al)$ into an orthogonal direct sum of irreducibles $\pi_1,\ldots,\pi_k \in \Irr(\cC_1)$ and take standard solutions $s_{\pi_i}$, $t_{\pi_i}$ of the conjugate equations for $\pi_i$ such that
$$\cF(s_\al) = \sum_{i=1}^k \lambda_i s_{\pi_i} \; ,$$
where $\lambda_1,\ldots,\lambda_k \in (0,+\infty)$. It follows that $\cF(t_\al) = \sum_{i=1}^k \lambda_i^{-1} t_{\pi_i}$. Since an injective $*$-homomorphism between C$^*$-algebras is isometric, we know that $\cF$ is isometric on spaces of morphisms. It follows that
$$d(\al) = \|\cF(s_\al)\|^2 = \sum_{i=1}^k \lambda_i^2 d_1(\pi_i) \quad\text{and}\quad
d(\al) = \|\cF(t_\al)\|^2 = \sum_{i=1}^k \lambda_i^{-2} d_1(\pi_i) \; .$$
Since $\cF$ is dimension preserving, we also have that $d(\al) = \sum_{i=1}^k d_1(\pi)$. Since $\lambda_i^2 + \lambda_i^{-2} \geq 2$ for all $i$, we conclude that $\lambda_i = 1$ for all $i=1,\ldots,k$. This exactly means that $\cF(s_\al)$, $\cF(t_\al)$ is a standard solution of the conjugate equations for $\cF(\al)$.

Let $\vphi_1 : \Irr(\cC_1) \recht \C$ be a cp-multiplier. Define the corresponding positive functional $\om_1 \in C_u(\cC_1)^*$ given by $\om_1(\pi) = d_1(\pi) \vphi_1(\pi)$ for all $\pi \in \Irr(\cC_1)$. Define $\vphi : \Irr(\cC) \recht \C : \vphi(\al) = d(\al)^{-1} \om_1(\cF(\al))$. We have to prove that $\vphi$ is a cp-multiplier on $\cC$.

For every $\al,\be \in \cC$, we denote by $E_{\al,\be}$ the unique trace preserving conditional expectation of $\End(\cF(\al) \ot \cF(\be))$ onto the $*$-subalgebra $\cF(\End(\al \ot \be))$. We then define the cp maps
$$\theta_{\al,\be} : \End(\al \ot \be) \recht \End(\al \ot \be) : \theta_{\al,\be} = \cF^{-1} \circ E_{\al,\be} \circ \theta^{\vphi_1}_{\cF(\al),\cF(\be)} \circ \cF \; .$$
Since $\cF$ preserves the standard solutions of the conjugate equations, we get that
$$E_{\al_1 \ot \al, \be \ot \be_1}(1 \ot T \ot 1) = 1 \ot E_{\al,\be}(T) \ot 1 \quad\text{for all}\;\; T \in \End(\cF(\al) \ot \cF(\be)) \; .$$
It follows that $\theta_{\al,\be}$ defines a cp-multiplier on $\cC$. So, $\theta_{\al,\be} = \theta^\psi_{\al,\be}$ for some cp-multiplier $\psi : \Irr(\cC) \recht \C$. It suffices to prove that $\vphi = \psi$.

Take $\al \in \Irr(\cC)$. As explained above, we can decompose $\cF(\al)$ into a direct sum of irreducibles $\pi_1,\ldots,\pi_k \in \Irr(\cC_1)$ and choose standard solutions of the conjugate equations $s_{\pi_i}, t_{\pi_i}$ such that
\begin{equation}\label{eq.myformula}
\cF(s_\al) = \sum_{i=1}^k s_{\pi_i} \; .
\end{equation}
It follows that
\begin{equation}\label{eq.okokok}
\theta^{\vphi_1}(\cF(s_\al)) = \sum_{i=1}^k \vphi_1(\pi_i) s_{\pi_i} \; .
\end{equation}
We must compute the orthogonal projection of the right hand side of \eqref{eq.okokok} on the $1$-dimensional space of multiples of $\cF(s_\al)$. Using \eqref{eq.myformula}, this orthogonal projection equals
$$\frac{1}{d(\al)} \Bigl( \sum_{i=1}^k d_1(\pi_i) \vphi_1(\pi_i)\Bigr) \; \cF(s_\al) \; .$$
We conclude that
$$\psi(\al) = \frac{1}{d(\al)} \sum_{i=1}^k d_1(\pi_i) \vphi_1(\pi_i) = \frac{1}{d(\al)} \om_1\Bigl(\sum_{i=1}^k \pi_i \Bigr) = \frac{\om_1(\cF(\al))}{d(\al)} = \vphi(\al) \; ,$$
because $\cF(\al)$ is isomorphic to the direct sum of $\pi_1,\ldots,\pi_k$.
\end{proof}

We can then prove that property~(T) passes to quotients.

\begin{proposition}\label{prop.quotient-T}
Let $\cC$, $\cC_1$ be rigid C$^*$-tensor categories and $\cF : \cC \recht \cC_1$ a dimension preserving C$^*$-tensor functor. Assume that $\cF$ is ``surjective'' in the following sense~: every $\pi \in \Irr(\cC_1)$ appears in some $\cF(\al)$, $\al \in \cC$. If $\cC$ has property~\pT, then also $\cC_1$ has property~\pT.
\end{proposition}
\begin{proof}
By Proposition \ref{prop.char-prop-T}, we have the nonzero projection $p \in C_u(\cC)$ satisfying $\al p = d(\al) p$ for all $\al \in \cC$. By Proposition \ref{prop.quotient-hom}, the map $\al \mapsto \cF(\al)$ extends to a unital $*$-homomorphism $\Ups : C_u(\cC) \recht C_u(\cC_1)$. Define $q = \Ups(p)$. Then $q$ is a projection in $C_u(\cC_1)$. Since $\cF$ is dimension preserving, we have $\counit_1 \circ \Ups = \counit$. Therefore, $\counit_1(q) = \counit(p) = 1$, so that $q$ is nonzero. For every $\al \in \cC$, we have $\Ups(\al) q = \Ups(\al p) = d(\al) q$. It follows that $\pi q = d_1(\pi) q$ for every subobject $\pi$ of $\cF(\al)$. By the surjectivity assumption, we conclude that $\pi q = d_1(\pi) q$ for every $\pi \in \cC_1$. By Proposition \ref{prop.char-prop-T}, this means that $\cC_1$ has property~(T).
\end{proof}

Applying Proposition \ref{prop.quotient-T} to the case of bimodule categories arising from subfactors and taking into account the comments before Proposition \ref{prop.quotient-hom}, we can thus reprove \cite[Theorem 9.9]{Po99}.

\begin{corollary}
Let $\cG$ be a $\lambda$-lattice and $\cG_0$ a sub-$\lambda$-lattice of $\cG$. If $\cG_0$ has property~(T), then also $\cG$ has property~(T).
\end{corollary}
\begin{proof}
By \cite{Po94b} (see also \cite[Theorem 4.3]{Po00}), we can realize $\cG_0 = \cG_{N_0,M_0}$ and $\cG = \cG_{N,M}$ where $N_0 \subset M_0$ and $N \subset M$ are extremal subfactors with Jones towers $N_0 \subset M_0 \subset M_{0,1} \subset \cdots$ and $N \subset M \subset M_1 \subset \cdots$ that fit into a nondegenerate commuting square
$$\coms{N}{M}{N_0}{M_0}$$
with the property that $N_0' \cap M_{0,n} \subset N' \cap M_n$ for all $n \geq 0$. Then the map that sends the $M_0$-bimodule $L^2(M_{0,n})$ to the $M$-bimodule $L^2(M_n)$ uniquely extends to a dimension preserving C$^*$-tensor functor $\cF : \cC_0 \recht \cC$, with $\cC_0$ defined as the category of $M_0$-bimodules generated by $N_0 \subset M_0$ and $\cC$ as the category of $M$-bimodules generated by $N \subset M$. Since the image of $\cF$ contains $\bim{M}{L^2(M_1)}{M}$, we get that $\cF$ is ``surjective'' in the sense of Proposition \ref{prop.quotient-T}. Combining Propositions \ref{prop.all-equiv} and \ref{prop.quotient-T}, the result follows.
\end{proof}

\begin{remark}\label{rem.essential-preserve-salpha}
\begin{enumlist}
\item In Proposition \ref{prop.quotient-T}, it is essential to assume that $\cF$ is dimension preserving. Indeed, restricting finite dimensional unitary representations of $\SU_q(3)$ to the maximal torus $\T^2$ yields a C$^*$-tensor functor $\cF : \Rep(\SU_q(3)) \recht \Rep(\T^2)$. However, $\Rep(\SU_q(3))$ has property~(T) (by \cite{Ar14}, see Theorem \ref{thm.main-SUq3}), while $\Rep(\T^2)$ does not (it is an amenable C$^*$-tensor category with infinitely many irreducible objects).

\item The TLJ $\lambda$-lattice is a sub-$\lambda$-lattice of an arbitrary $\lambda$-lattice. The C$^*$-tensor category counterpart of this observation goes as follows. For every $d \in \R$ with $|d| \geq 2$, we denote by $\cC_d$ the Temperley-Lieb C$^*$-tensor category. Then, $\cC_d$ is isomorphic with $\Rep(\SU_q(2))$ where $q \in [-1,1] \setminus \{0\}$ is given by $q+1/q = - d$. Also, $\cC_d$ can be defined by completing the category having the natural numbers as objects, with $\Mor(n,m)=\{0\}$ if $n-m$ is odd and with $\Mor(n,m)$ being the vector space having as a basis the non-crossing pairings between $n$ points and $m$ points. When two such non-crossing pairings are composed, all closed loops are replaced by the scalar $d$.

    Let now $\cC$ be an arbitrary rigid C$^*$-tensor category and $\al \in \cC$ a real, resp.\ pseudo-real, object. This means that $\al \cong \albar$ and that we can choose a standard solution for the conjugate equations $t,s \in \Mor(\al \ot \al,\eps)$ in which $t = s$, resp.\ $t = -s$. Putting $d = d(\al)$, resp.\ $d = -d(\al)$, there is a unique dimension preserving C$^*$-tensor functor $\cF : \cC_d \recht \cC$ that maps the non-crossing pairing $\cup \in \Mor(2,0)$ to $t$.

    When $\cC$ is the bimodule category generated by an extremal subfactor $N \subset M$, then $\al = \bim{M}{L^2(M_1)}{M}$ is a real object in $\cC$ with dimension $\lambda^{-1} = [M:N]$. The above functor $\cF : \cC_{\lambda^{-1}} \recht \cC$ is then the precise counterpart of $\cG^\lambda$ being a sub-$\lambda$-lattice of $\cG_{N,M}$.

\item In point 2, note that $\cF$ is ``surjective'' in the sense of Proposition \ref{prop.quotient-T} if every $\be \in \Irr(\cC)$ is a subobject of some tensor power of $\al$. So in the example of a subfactor bimodule category with $\al = \bim{M}{L^2(M_1)}{M}$, this is always the case. This example shows that in general, the surjectivity of a dimension preserving C$^*$-tensor functor $\cF : \cC \recht \cC_1$ does not imply the surjectivity of the associated $*$-homomorphism $C_u(\cC) \recht C_u(\cC_1)$ constructed in Proposition \ref{prop.quotient-hom}.
\end{enumlist}
\end{remark}

\section{Concluding remarks}\label{sec.different-Cstar-algebras}

Let $\cC$ be a rigid C$^*$-tensor category. The fusion rules of $\cC$, which show how to decompose $\al \ot \be$ into a direct sum of irreducibles, turn $\Irr(\cC)$ into a hypergroup. When $\cC$ is the category of $M$-bimodules generated by a finite index subfactor $N \subset M$, this hypergroup structure is essentially the same as the (even part of) the principal graph of the subfactor $N \subset M$. The definition of the fusion $*$-algebra $\C[\cC]$ only uses the hypergroup structure on $\Irr(\cC)$, i.e.\ the fusion rules. Also the regular representation of $\C[\cC]$ on $\ell^2(\cC)$ is defined entirely in terms of the fusion rules, and hence so is the reduced C$^*$-algebra $C_r(\cC)$. Thus, in the case where $\cC$ is the category of $M$-bimodules generated by a subfactor $N \subset M$, all these objects are entirely determined by the principal graph of the subfactor.

However, the definition of the universal C$^*$-algebra $C_u(\cC)$ uses much more of the structure of $\cC$ and we strongly believe that $C_u(\cC)$ cannot be defined only in terms of the fusion rules and the dimension function on $\Irr(\cC)$ (or, equivalently, in terms of the weighted principal graph when $\cC$ comes from a subfactor). In particular, it now seems entirely possible that there do exist two $\lambda$-lattices $\cG_1$, $\cG_2$ with the same weighted principal graph, but such that $\cG_1$ has property~(T), resp.\ the Haagerup property, while $\cG_2$ does not (compare with \cite[Remark 9.11]{Po99} where it is conjectured that the opposite might be true). However, we do not have examples of such phenomena and there are several reasons why it is difficult to find them.

\begin{enumlist}
\item In the amenable case, by Proposition \ref{prop.equiv-amen}, $C_u(\cC) \cong C_r(\cC)$ and thus, $C_u(\cC)$ only depends on the fusion rules and the dimension function.

\item As observed by Kenny De Commer, when $\cC_1$ and $\cC_2$ both have the fusion rules of the compact Lie group $\SU(n)$ and have the same dimension function, then the C$^*$-algebras $C_u(\cC_i)$ are isomorphic because by \cite{Jo14}, the C$^*$-tensor categories $\cC_i$ are a twist of the same $\Rep(\SU_q(n))$ by a scalar $3$-cocycle. In particular, when $n$ is an odd integer, $n \geq 3$, every rigid C$^*$-tensor category with the same fusion rules as $\SU(n)$ is either amenable, or has property~(T).

\item When $\Gamma$ is a hypergroup with unit and dimension function $d : \Gamma \recht [1,\infty)$, we can define the $L^1$-norm on $\C[\Gamma]$ by $\|x\|_1 = \sum_{\al \in \Gamma} d(\al) \, |x_\al|$. As such, one obtains the unital Banach $*$-algebra $L^1(\Gamma,d)$. If now $\Gamma = \Irr(\cC)$ and $d$ is given by the categorical dimension, it is tempting to guess that $C_u(\cC)$ is the universal enveloping C$^*$-algebra of $L^1(\Gamma,d)$. This is however not always true. More precisely, it fails when $\cC$ is given by the TLJ $\lambda$-lattice $\cG^\lambda$ with $\lambda^{-1} > 4$.

    Indeed, using the notation of Theorem \ref{thm.main-temperley-lieb}, we identify $\C[\cC]$ with the polynomial $*$-algebra $\C[X]$ with $X^* = X$, and with $X$ corresponding to the $M$-bimodule $L^2(M_1)$. For every $t \in \R$, we have the unique $*$-homomorphism $\counit_t : \C[\cC] \recht \C : \counit_t(L^2(M_1)) = t$. Then $\counit_t$ defines a continuous $*$-homomorphism on $L^1(\Irr(\cC),d)$ if and only if $|\counit_t(H_n)| \leq d(H_n)$ for every $n \in \N$. This condition exactly means that $|V_n(t)| \leq V_n(\lambda^{-1})$ for all $n \in \N$.

    Using the recurrence relation \eqref{eq.recurrence-Vn}, we get that
    \begin{align*}
    V_n(2(1+\cos \al)) &= \frac{\sin((n+1)\al) + \sin(n \al)}{\sin(\al)} \;\; , \\
    V_n(2(1+\cosh \al)) &= \frac{\sinh((n+1)\al) + \sinh(n \al)}{\sinh(\al)} \;\; , \\
    V_n(2(1-\cosh \al)) &= (-1)^n \frac{\sinh((n+1)\al) - \sinh(n \al)}{\sinh(\al)} \;\; ,
    \end{align*}
    for all $n \in \N$, $\al \in \R$. It is then easy to check that $|V_n(t)| \leq V_n(\lambda^{-1})$ for all $n \in \N$ if and only if $t \in [4-\lambda^{-1},\lambda^{-1}]$. So the universal enveloping C$^*$-algebra of $L^1(\Irr(\cC),d)$ can be identified with $C([4-\lambda^{-1},\lambda^{-1}])$. When $\lambda^{-1} > 4$, we indeed find that the natural $*$-homomorphism onto $C_u(\cC)=C([0,\lambda^{-1}])$ is not faithful.

\item  When $\cC$ is the natural C$^*$-tensor category with $\Irr(\cC) = \Gamma$, for a countable group $\Gamma$, we have that $C_u(\cC)$ coincides with the universal enveloping C$^*$-algebra of $L^1(\Irr(\cC),d)$ and $C_u(\cC) \cong C^*(\Gamma)$.
\end{enumlist}

\end{document}